\documentclass[12pt]{article}

\usepackage{amscd}
\usepackage{amsmath}
\usepackage{amssymb}
\usepackage{amstext}
\usepackage{amsthm}
\usepackage{bbold}
\usepackage{bm}
\usepackage{booktabs}
\usepackage{color}
\usepackage{colonequals}
\usepackage{comment}

\usepackage{easybmat}
\usepackage{etex}
\usepackage{enumitem}
\usepackage{framed}
\usepackage{authblk}
\usepackage[dvips,letterpaper,margin=1in]{geometry}
\usepackage{graphicx}
\usepackage{listings}
\usepackage{longtable}
\usepackage{mathtools}
\usepackage{multirow}
\usepackage{rotating}
\usepackage{setspace}
\usepackage{siunitx}
\usepackage{tabu}
\usepackage{verbatim}

\definecolor{cblack}{rgb}{0,0,0}
\definecolor{cblue}{rgb}{0.121569,0.466667,0.705882}    % 31,  119, 180
\definecolor{corange}{rgb}{1.000000,0.498039,0.054902}  % 256, 127, 14
\definecolor{cgreen}{rgb}{0.172549,0.627451,0.172549}   % 44,  160, 44
\definecolor{cred}{rgb}{0.839216,0.152941,0.156863}     % 214, 39,  40
\definecolor{cpurple}{rgb}{0.580392,0.403922,0.741176}  % 149, 103, 190
\definecolor{cbrown}{rgb}{0.549020,0.337255,0.294118}   % 141, 86,  75
\definecolor{cpink}{rgb}{0.890196,0.466667,0.760784}
\definecolor{cgray}{rgb}{0.498039,0.498039,0.498039}
\definecolor{cgreen2}{rgb}{0.7372549019607844, 0.7411764705882353, 0.13333333333333333}

\usepackage{hyperref}
\hypersetup{
  linkcolor  = cblue,
  citecolor  = cgreen,
  urlcolor   = corange,
  colorlinks = true,
}

\newtheorem{theorem}{Theorem}[section]
\newtheorem{remark}[theorem]{Remark}

\newtheorem{lemma}[theorem]{Lemma}

\newtheorem{definition}[theorem]{Definition}
\newtheorem{example}[theorem]{Example}
\newtheorem{proposition}[theorem]{Proposition}
\newtheorem{corollary}[theorem]{Corollary}

\theoremstyle{plain} % just in case the style had changed
\newcommand{\thistheoremname}{}
\newtheorem*{genericthm}{\thistheoremname}

\newcommand{\what}{\widehat}

\makeatletter
\def\moverlay{\mathpalette\mov@rlay}
\def\mov@rlay#1#2{\leavevmode\vtop{%
   \baselineskip\z@skip \lineskiplimit-\maxdimen
   \ialign{\hfil$\m@th#1##$\hfil\cr#2\crcr}}}
\newcommand{\charfusion}[3][\mathord]{
    #1{\ifx#1\mathop\vphantom{#2}\fi
        \mathpalette\mov@rlay{#2\cr#3}
      }
    \ifx#1\mathop\expandafter\displaylimits\fi}
\makeatother

\newcommand{\EE}{\mathbb{E}}

\newcommand{\RR}{\mathbb{R}}

\DeclareSymbolFont{bbold}{U}{bbold}{m}{n}
\DeclareSymbolFontAlphabet{\mathbbold}{bbold}
\newcommand{\One}{\mathbbold{1}}

\newcommand{\ba}{\bm a}
\newcommand{\bb}{\bm b}

\newcommand{\be}{\bm e}

\newcommand{\bv}{\bm v}
\newcommand{\bw}{\bm w}
\newcommand{\bx}{\bm x}
\newcommand{\by}{\bm y}
\newcommand{\bz}{\bm z}
\newcommand{\bA}{\bm A}
\newcommand{\bB}{\bm B}

\newcommand{\bM}{\bm M}

\newcommand{\bP}{\bm P}

\newcommand{\bV}{\bm V}

\newcommand{\bY}{\bm Y}
\newcommand{\bZ}{\bm Z}
\newcommand{\one}{\bm{1}}

\newcommand{\sH}{\mathcal{H}}
\newcommand{\sI}{\mathcal{I}}

\newcommand{\sM}{\mathcal{M}}

\DeclareSymbolFont{sfoperators}{OT1}{cmss}{m}{n}
\DeclareSymbolFontAlphabet{\mathsf}{sfoperators}
\makeatletter
\renewcommand{\operator@font}{\mathgroup\symsfoperators}
\makeatother

\DeclareMathOperator{\Gram}{Gram}

\DeclareMathOperator{\sym}{sym}
\DeclareMathOperator{\Part}{Part}

\DeclareMathOperator{\Sym}{Sym}

\newcommand{\Ex}{\mathop{\mathbb{E}}}  % puts symbols below
\newcommand{\vol}{\mathsf{vol}}

\newcommand{\tEE}{\widetilde{\mathbb{E}}}

\title{The spectrum of the Grigoriev-Laurent pseudomoments}
\date{March 10, 2022}

\usepackage{authblk}
\author[1]{Dmitriy Kunisky\thanks{Email: \textit{dmitriy.kunisky@yale.edu}. Partially supported by ONR Award N00014-20-1-2335, a Simons Investigator Award to Daniel Spielman, and NSF grants DMS-1712730 and DMS-1719545. Part of this work was performed while DK was with New York University.}}
\author[2]{Cristopher Moore\thanks{Email: \textit{moore@santafe.edu}. Partially supported by NSF grant IIS-1838251.}}

\affil[1]{Department of Computer Science, Yale University}
\affil[2]{Santa Fe Institute}

\begin{document}

\maketitle
\thispagestyle{empty}

\begin{abstract}
    Grigoriev (2001) and Laurent (2003) independently showed that the sum-of-squares hierarchy of semidefinite programs does not exactly represent the hypercube $\{\pm 1\}^n$ until degree at least $n$ of the hierarchy.
    Laurent also observed that the pseudomoment matrices her proof constructs appear to have surprisingly simple and recursively structured spectra as $n$ increases.
    While several new proofs of the Grigoriev-Laurent lower bound have since appeared, Laurent's observations have remained unproved.
    We give yet another, representation-theoretic proof of the lower bound, which also yields exact formulae for the eigenvalues of the Grigoriev-Laurent pseudomoments.
    Using these, we prove and elaborate on Laurent's observations.

    Our arguments have two features that may be of independent interest.
    First, we show that the Grigoriev-Laurent pseudomoments are a special case of a Gram matrix construction of pseudomoments proposed by Bandeira and Kunisky (2020).
    Second, we find a new realization of the irreducible representations of the symmetric group corresponding to Young diagrams with two rows, as spaces of multivariate polynomials that are \emph{multiharmonic} with respect to an equilateral simplex.
\end{abstract}

\clearpage

\tableofcontents

\pagestyle{empty}

\clearpage

\setcounter{page}{1}
\pagestyle{plain}

\section{Introduction}

\subsection{Sum-of-Squares and the Grigoriev-Laurent Lower Bound}

The sum-of-squares (SOS) hierarchy is a powerful family of semidefinite programming (SDP) algorithms for computing bounds on polynomial optimization problems \cite{Shor-1987-SumOfSquares, Nesterov-2000-SOS, Lasserre-2001-GlobalOptimizationMoments, Parrilo-2003-SDPSemialgebraic}.
Because of the power of these algorithms for many theoretical problems (see, e.g., \cite{BS-2014-SOSQuest} for a survey and \cite{FKP-2019-SemialgebraicProofsAlgorithm} for a monograph treatment), proving \emph{lower bounds} showing that SOS programs do not give tight bounds on various problems has become an important direction in theoretical computer science \cite{Grigoriev-2001-SOSParity,Schoenebeck-2008-SOSCSP,MPW-2015-PlantedClique,KMOW-2017-SOSCSP,BHKKMP-2019-PlantedClique,GJJPR-2020-SK,PR-2020-MachineryPseudocalibration}.

In this paper, we revisit an early result on the SOS relaxation of optimizing a polynomial over the Boolean hypercube.
We first review the definition of this relaxation; our discussion follows as a special case of the general framework presented in, e.g., the survey \cite{Laurent-2009-SOS}.

We fix a few basic notations: we adopt the standard $[n] \colonequals \{1, \dots, n\}$, write $\binom{S}{k}$ and $\binom{S}{\leq k}$ for the sets of subsets of a set $S$ with exactly $k$ and at most $k$ elements, respectively, and, for a set $S \subseteq [n]$ and $\bx \in \RR^n$, write $\bx^S \colonequals \prod_{i \in S} x_i$.
We also write $\RR[x_1, \dots, x_n]_{\leq d}$ for the set of polynomials of degree at most $d$.
\begin{definition}[Hypercube pseudoexpectation]
    \label{def:pe}
    We say $\tEE: \RR[x_1, \dots, x_n]_{\leq 2d} \to \RR$ is a \emph{degree~$2d$ pseudoexpectation}\footnote{Usually we should specify a pseudoexpectation ``over $\{ \pm 1\}^n$'' or, more precisely yet, ``with respect to the constraints $x_i^2 - 1 = 0$,'' but we will only work over the hypercube in this paper so we omit these specifications.} if the following conditions hold:
    \begin{enumerate}
    \item $\tEE$ is linear,
    \item $\tEE[1] = 1$,
    \item $\tEE[(x_i^2 - 1) p(\bx)] = 0$ for all $i \in [n]$, $p \in \RR[x_1, \dots, x_n]_{\leq 2d - 2}$,
    \item $\tEE[p(\bx)^2] \geq 0$ for all $p \in \RR[x_1, \dots, x_n]_{\leq d}$.
    \end{enumerate}
\end{definition}
\noindent
For $\mu$ a probability measure over $\{\pm 1\}^n$, let $\EE_{\mu}$ denote the expectation operator with respect to $\mu$.
Then, any $\EE_{\mu}$ is a pseudoexpectation of any degree; however, some pseudoexpectations do not arise in this way.
The use of Definition~\ref{def:pe} in computation is that, so long as $\deg p \leq 2d$, using this observation we may bound a polynomial optimization problem by
\begin{equation}
    \max_{\bx \in \{\pm 1\}^n} p(\bx) = \max_{\substack{\mu \text{ a probability} \\ \text{measure over } \{\pm 1\}^n}} \EE_{\mu}[p(\bx)] \leq \max_{\substack{\tEE \text{ a degree } 2d \\ \text{pseudoexpectation}}} \tEE[p(\bx)],
\end{equation}
and the right-hand side may be computed in time $n^{O(d)}$ by solving a suitable SDP (this was taken for granted in earlier works like \cite{Lasserre-2001-GlobalOptimizationMoments, Parrilo-2003-SDPSemialgebraic, Laurent-2009-SOS}; more recently \cite{ODonnell-2017-SOSNotAutomatizable} noticed an important and previously neglected technicality, which is handled for our specific setting by \cite{RW-2017-BitComplexity}).

How effective are these relaxations?
The following result of Laurent \cite{Laurent-2003-CutPolytopeSOS}, one of the first strong lower bounds proved against the sum-of-squares hierarchy, shows that the relaxation is not \emph{tight}---achieving equality above for all polynomials $p$---until the very high degree $d \approx n$.
Essentially the same result is also latent in the slightly earlier work of Grigoriev \cite{Grigoriev-2001-SOSKnapsack}, but it will be more convenient for us to follow Laurent's discussion, phrased over the $\{\pm 1\}^n$ hypercube, than Grigoriev's, phrased over the $\{0, 1\}^n$ hypercube.
\begin{theorem}[Theorem 6 of \cite{Laurent-2003-CutPolytopeSOS}]
    \label{thm:laurent-pe}
    For any $n \geq 3$ odd, there exists $\tEE$ a degree $n - 1$ pseudoexpectation such that $\tEE \neq \EE_{\mu}$ for any probability measure $\mu$ over $\{\pm 1\}^n$.
\end{theorem}
\noindent
We note that this result is sharp, in the sense that for any $\tEE$ a degree $n + 1$ pseudoexpectation, there does exist $\mu$ such that $\tEE = \EE_{\mu}$, as conjectured by \cite{Laurent-2003-CutPolytopeSOS} and later proved by \cite{FSP-2016-SOSLifts}.

We next present in simple linear-algebraic terms the main technical claim underlying Laurent's proof.
\begin{definition}[Grigoriev-Laurent pseudomoments]
    Define the constants
    \begin{equation}
        \label{eq:as}
        a_k \colonequals \One\{k \text{ even}\} \cdot (-1)^{k / 2} \prod_{i = 0}^{k / 2 - 1}\frac{2i + 1}{n - 2i - 1}.
    \end{equation}
    For $n \geq 2$, set
    \begin{equation}
        d_{\max} = d_{\max}(n) \colonequals \left\lfloor \frac{n}{2}\right \rfloor.
    \end{equation}
    Define the matrix $\bY^{(n)} \in \RR^{\binom{[n]}{\leq d_{\max}} \times \binom{[n]}{\leq d_{\max}}}$ to have entries
    \begin{equation}
        Y^{(n)}_{S,T} \colonequals a_{|S \triangle T|},
    \end{equation}
    where $S \triangle T$ denotes the symmetric difference of sets.
\end{definition}

\begin{theorem}[Theorem 6 of \cite{Laurent-2003-CutPolytopeSOS}, rephrased]
    \label{thm:laurent-pm}
    For any $n \geq 2$, $\bY^{(n)} \succeq \bm 0$.
\end{theorem}
\noindent
Indeed, the pseudoexpectation in Theorem~\ref{thm:laurent-pe} has $\tEE[\bx^S] = a_{|S|}$ for all sets $S \subseteq [n]$ and extends to all polynomials by linearity and the property that $\tEE[x_i^2p(\bx)] = \tEE[p(\bx)]$.
The matrix $\bY^{(n)}$ is then called the \emph{pseudomoment matrix} of $\tEE$, and its entries the \emph{pseudomoments}.
With this choice, all properties in Definition~\ref{def:pe} may be readily verified except the last, which is equivalent to the positivity claimed in Theorem~\ref{thm:laurent-pm} (this equivalence is also why SOS optimization problems may be solved with SDPs).
Further, $\tEE \neq \EE_{\mu}$ for any probability measure $\mu$ when $n$ is odd, since a computation shows that $\tEE[(\sum x_i)^2] = 0$, while $\EE_{\mu}[(\sum x_i)^2] \geq 1$ for all $\mu$ because $\sum x_i$ is an odd integer.
We elaborate further on how this argument motivates the choice of the pseudomoments of degree greater than two in Section~\ref{sec:prelim-pseudomoments}.

The matrix $\bY^{(n)}$ decomposes as the direct sum of two principal submatrices, those indexed by $\binom{[n]}{d}$ with $d$ even and odd respectively, which Laurent considers separately, but it will be more natural in our calculations to avoid this decomposition.\footnote{In fact, without loss of generality we could from the outset restrict our attention to pseudomoment matrices that factorize in this way, since for any pseudoexpectation $\tEE$ we may work equally well with $\tEE^{\prime}[p(\bx)] \colonequals \frac{1}{2}(\tEE[p(\bx)] + \tEE[p(-\bx)])$, whose pseudomoments of odd degree are zero. This reduction corresponds to the invariance of the hypercube constraints and the quadratic form $(\sum x_i)^2$ under the mapping $\bx \mapsto -\bx$.}
Adjusting for this minor change, Laurent's proof may be seen as identifying a $\binom{n}{\leq d_{\max} - 1}$-dimensional kernel of $\bY^{(n)}$, and proving that the principal submatrix $\bZ^{(n)}$ of $\bY^{(n)}$ indexed by $\binom{[n - 1]}{d_{\max}} \cup \binom{[n - 1]}{d_{\max} - 1}$, which has total dimension $\binom{n}{d_{\max}}$, is strictly positive definite.
That $\bY^{(n)} \succeq \bm 0$ then follows by interlacing of eigenvalues.
While identifying the kernel is straightforward, for the second part of the argument Laurent uses that each block of $\bZ^{(n)}$ belongs to the Johnson association scheme, and applies formulae for the eigenvalues of the matrices spanning the Johnson scheme's Bose-Mesner algebra.
Concretely, this expresses the eigenvalues of $\bZ^{(n)}$ as combinatorial sums involving binomial coefficients and having alternating signs.
To establish positivity, Laurent then uses general identities for transforming hypergeometric series \cite{PWZ-1996-AB}, which yield different expressions for the eigenvalues of $\bZ^{(n)}$ as sums of only positive terms.

Laurent also makes several further empirical observations about $\bY^{(n)}$ and its ``quite remarkable structural properties'' in the appendix of \cite{Laurent-2003-CutPolytopeSOS}, which we restate in Theorem~\ref{thm:laurent-evals} below.
Most notably, as $n$ increases, the spectrum of $\bY^{(n)}$ appears to ``grow'' in a simple recursive fashion, with the eigenvalues of $\bY^{(n + 2)}$ equaling those of $\bY^{(n)}$ multiplied by $\frac{n + 2}{n + 1}$, along with a new largest eigenvalue.
Unfortunately, the proof outlined above does not make use of this elegant structure and does not give any indication of why it should hold.
These intriguing observations have remained unproved to date.

\subsection{Main Result and Proof Ideas}

Our main theorem proves Laurent's observations, along with some further details of the recursion she proposed (namely, an explicit formula for the base case---the ``new'' largest eigenvalue alluded to above---and formulae for the eigenvalue multiplicities).

\begin{theorem}
    \label{thm:laurent-evals}
    $\bY^{(n)}$ has $d_{\max} + 2$ distinct eigenvalues, $0 < \lambda_{n, d_{\max}} < \cdots < \lambda_{n, 1} < \lambda_{n, 0}$.
    The multiplicity of the zero eigenvalue is $\binom{n}{\leq d_{\max} - 1}$, while the $\lambda_{n, d}$ have the following multiplicities and recursive description:
    \begin{align*}
        \lambda_{n, 0} &= \sum_{k = 0}^{d_{\max}} \binom{n}{k}a_k^2 \text{ with multiplicity } 1, \\
        \lambda_{n, d} &= \frac{n}{n - 1} \lambda_{n - 2, d - 1} \text{ for } 1 \leq d \leq (n - 1) / 2 \text{ with multiplicity } \binom{n}{d} - \binom{n}{d - 1}.
    \end{align*}
    The following closed form also holds for $0 \leq d \leq d_{\max}$:
    \[ \lambda_{n, d} = n! \sum_{k = d}^{d_{\max}} \frac{a_{k - d}^2}{(n - d - k)!(k - d)!} \prod_{i = 0}^{d - 1} \frac{1}{(n - 2i - 1 - k + d)^2}. \]
\end{theorem}
\noindent
We will reach this result in two stages: first, we will give a new proof of Theorem~\ref{thm:laurent-pm}, based on some previous observations of \cite{BGP-2016-SOSHypercube}.
This proof will be a representation-theoretic one, and will give an abstract description of the eigenspaces and eigenvalues of $\bY^{(n)}$.
This description will suffice to prove positivity, but not to compute the eigenvalues as explicitly as we wish to.

To complete that computation, we will show moreover that $\bY^{(n)}$ and the associated pseudoexpectation $\tEE$ arise from a particular case of the \emph{spectral extension} construction proposed by \cite{KB-2020-Degree4SK, Kunisky-2020-PositivityPreservingExtensions} which describes $\bY^{(n)}$ as a Gram matrix of certain polynomials under the \emph{apolar inner product}.
This observation will allow us to rephrase the computation of the eigenvalues $\bY^{(n)}$ in terms of this Hilbert space of polynomials, which we can carry out in closed form.
The basic idea is that, after writing $\bY^{(n)} = \bA^{\top}\bA$ using spectral extension, we will use that $\bY^{(n)}$ has the same non-zero spectrum as $\bA\bA^{\top}$.
It turns out that basic representation-theoretic reasoning about symmetries of the ``Gram vectors'' that are the columns of $\bA$ implies that the latter is actually a diagonal matrix.
To identify the spectrum of $\bY^{(n)}$, it then remains only to extract the diagonal entries and count their multiplicities, which is a non-trivial but tractable combinatorial calculation.

We note also that Laurent observed that it appears plausible to prove $\bY^{(n)} \succeq \bm 0$ by repeatedly taking Schur complements with respect to blocks indexed by subsets of fixed size $\binom{[n]}{d}$, in the order $d = 0, 1, \dots, d_{\max}$.
This would give a perhaps more conceptually-satisfying proof than the original one relying on eigenvalue interlacing---which offers no direct insight into the spectrum of $\bY^{(n)}$ itself---but appears quite technical to conduct.
Our approach will, in showing that $\bY^{(n)}$ is a special case of spectral extension, implicitly carry out this plan, as we discuss later in Section~\ref{sec:schur-complement}.

\subsection{Related Work}

At least three other, conceptually different, proofs of the Grigoriev-Laurent lower bound (that we know of) have since appeared.
First, \cite{KLM-2016-SymmetricFormulations} showed that, for highly symmetric problems over subsets of the hypercube, the positivity of the ``natural'' pseudoexpectation constructed from symmetry considerations reduces to a small number of univariate polynomial inequalities.
In the case of Laurent's result, these inequalities simplify algebraically and yield a proof, while in other cases this machinery calls for analytic arguments.

Second, \cite{BGP-2016-SOSHypercube} produced an elegant proof of a stronger result, showing that the function $(\sum_{i = 1}^n x_i)^2 - 1$ is not even a sum of squares of \emph{rational} functions of degree at most $d_{\max}$.
Their proof works in the dual setting, describing the decomposition of the space of functions on the hypercube into irreducible representations of the symmetric group and considering how a hypothetical sum of squares expression decomposes into associated components in these subspaces.
We will present this decomposition below and apply it at the beginning of our computations.

Third, \cite{Potechin-2017-GoodStory} showed that the result, in Grigoriev's form over the $\{0, 1\}^n$ hypercube, follows from another general representation-theoretic reduction of positivity conditions to a lower-dimensional space of polynomials.

Finally, we mention that many of our results appeared previously in preliminary form in the first author's dissertation \cite{Kunisky-2021-Thesis}.

\subsection{Notation}

We write $a \wedge b$ and $a \vee b$ for the minimum and maximum, respectively, of $a, b \in \RR$.
We write $\sM_k(S)$ for the set of multisets of size $k$ with elements belonging to $S$, and, as for sets, for $S \in \sM_k([n])$ and $\bx \in \RR^n$, we write $\bx^S \colonequals \prod_{i \in S} x_i$.
We write $\RR[x_1, \dots, x_n]^{\hom}_d$ for the polynomials in the given indeterminates that are homogeneous of degree $d$, and write $[x^k]p(x)$ for the coefficient of $x^k$ in $p$.
For a collection of vectors $\ba_1, \dots, \ba_m$ of equal dimension, we write $\Gram(\ba_1, \dots, \ba_m) \in \RR^{m \times m}$ for the Gram matrix:
    \begin{equation}
        \Gram(\ba_1, \dots, \ba_m)_{ij} \colonequals \langle \ba_i, \ba_j \rangle.
    \end{equation}

\section{Preliminaries}

\subsection{More on the Grigoriev-Laurent Pseudomoments}
\label{sec:prelim-pseudomoments}

We make a few further remarks on the pseudoexpectation $\tEE$ described above, which will play an important role later and also give some motivation for this construction.
We draw attention to the degree 2 pseudomoment matrix,
\begin{equation}
    \tEE[\bx\bx^{\top}] = \left[\begin{array}{cccc}
    1 & -\frac{1}{n - 1} & \cdots & -\frac{1}{n - 1} \\
    -\frac{1}{n - 1} & 1 & \cdots & -\frac{1}{n - 1} \\
    \vdots & \vdots & \ddots & \vdots \\
    -\frac{1}{n - 1} & -\frac{1}{n - 1} & \cdots & 1
    \end{array}\right] = \frac{n}{n - 1} \bm I_n - \frac{1}{n - 1} \one_n\one_n^{\top} \in \RR^{n \times n}_{\sym}.
\end{equation}
This is the Gram matrix of any $n$ unit vectors in $\RR^{n - 1}$ pointing to the vertices of an equilateral simplex with barycenter at the origin.
This choice of $\tEE[\bx\bx^{\top}]$ ensures that $\tEE[(\sum x_i)^2] = \one^{\top}\tEE[\bx\bx^{\top}] \one = 0$, which, as we saw above, is how we may verify that $\tEE \neq \EE_{\mu}$ for any probability measure $\mu$ when $n$ is odd.

Given this, it is not difficult to arrive at the values of the other $a_{k}$ defining the pseudomoments: we assume by symmetry that $\tEE[\bx^S]$ depends only on $|S|$; by symmetrizing $\tEE^{\prime}[p(\bx)] \colonequals \frac{1}{2}(\tEE[p(\bx)] + \tEE[p(-\bx)])$ we may assume that $\tEE[\bx^S] = 0$ whenever $|S|$ is odd; and we assume that not only does $\tEE[(\sum_{i = 1}^n x_i)^2] = 0$, but moreover that $\tEE[(\sum_{i = 1}^n x_i)p(\bx)] = 0$ whenever $\deg(p) \leq n - 2$ (sometimes called $\tEE$'s ``strongly satisfying'' the constraint $\sum_{i = 1}^n x_i = 0$).\footnote{Actually, one may verify that this property \emph{must} be satisfied by any $\tEE$ for which $\tEE[(\sum_{i = 1}^n x_i)^2] = 0$, using the ``SOS Cauchy-Schwarz inequality'': if $\tEE$ has degree $2d$ and $p, q \in \RR[x_1, \dots, x_n]_{\leq d}$, then $|\tEE[p(\bx) q(\bx)]| \leq (\tEE[p(\bx)^2])^{1/2} (\tEE[q(\bx)^2])^{1/2}$.}
Then, we must have $0 = \tEE[(\sum_{i = 1}^n x_i)\bx^S] = |S| a_{|S| - 1} + (n - |S|) a_{|S| + 1}$, and starting with $\tEE[1] = 1$ the values in \eqref{eq:as} follow recursively.
We also note that it is impossible to continue this construction past $|S| \leq (n - 1) / 2$ while retaining these properties, as solving the recursion calls for a division by zero.

An alternative way to motivate the construction of $\tEE$ is to note that, if $n$ is even and $\mu$ is the uniform measure over those $\bx \in \{\pm 1\}^n$ satisfying $\sum_{i = 1}^n x_i = 0$, then we have by straightforward counting arguments and manipulations of binomial coefficients
\begin{equation}
    \Ex_{\bx \sim \mu}[\bx^S] = \One\{|S| \text{ even}\} \frac{1}{\binom{n}{n / 2}} \sum_{k = 0}^{|S|} \binom{|S|}{k} \binom{n - |S|}{\frac{n}{2} - k} (-1)^k = \One\{|S| \text{ even}\}(-1)^{|S| / 2} \frac{\binom{n / 2}{|S| / 2}}{\binom{n}{|S|}}.
\end{equation}
Then, we recover our choice of $\tEE[\bx^S]$ if, even when $n$ is odd, we interpret the numerator above formally as
\begin{equation}
    \binom{n / 2}{m} \colonequals \frac{\frac{n}{2} \cdot (\frac{n}{2} - 1) \cdots (\frac{n}{2} - m + 1)}{m!} = \frac{n!!}{2^m m! (n - 2m)!!}.
\end{equation}
In this regard, the Grigoriev-Laurent lower bound shows that the SOS proof system requires high degree to distinguish between genuine binomial coefficients describing a counting procedure and their formal extension to rational inputs.\footnote{We thank Robert Kleinberg for teaching us this interpretation of the Grigoriev-Laurent lower bound.}

\subsection{Spectral Extensions of Pseudomoments}

Let us fix $\bv_1, \dots, \bv_n \in \RR^{n - 1}$ unit vectors pointing to the vertices of an equilateral simplex with $\sum \bv_i = \bm 0$.
As we have seen above, $\tEE[\bx\bx^{\top}]$ is then the Gram matrix of the $\bv_i$.

\begin{remark}
    We may produce a concrete realization of such simplex vectors by starting with $\be_1, \dots, \be_n \in \RR^n$, projecting these vectors to the orthogonal complement of $\one \in \RR^n$, and then expressing the resulting vectors in an orthonormal basis of $n - 1$ vectors for this orthogonal complement.
    For example, they may be realized as linear combinations of the standard basis $\be_1, \dots, \be_{n - 1}$ and the all-ones vector $\one$ in $\RR^{n - 1}$; we may take
    \begin{equation}
        \bv_i = \frac{1}{\sqrt{n - 1}}\left(\sqrt{n}\be_i - \frac{\sqrt{n} + 1}{n - 1}\one\right) \text{ for } i \in [n - 1] \text{ and } \bv_n = \frac{1}{\sqrt{n - 1}}\one.
    \end{equation}
\end{remark}

In such a setting, \cite{KB-2020-Degree4SK, Kunisky-2020-PositivityPreservingExtensions} proposed the technique of \emph{spectral extension} for determining the values of the higher-degree values of a pseudoexpectation $\tEE$.
Spectral extension proposes building $\tEE$ in a ``Gramian'' fashion, having $\tEE[\bx^S \bx^T] = \langle v_S, v_T \rangle$ for suitable vectors $v_S$.
Actually, these vectors may be interpreted as polynomials endowed with a particular inner product, which we review below. We follow the presentation of \cite{Reznick-1996-HomogeneousPolynomial}; see also the unpublished note \cite{Gichev-XXXX-HarmonicComponentHomogeneous}.

The inner product we use is defined as follows.
\begin{definition}[Apolar inner product]
For $p, q \in \RR[z_1, \dots, z_k]_{\hom}^d$, we define
\begin{equation}
    \langle p, q \rangle_{\circ} \colonequals \frac{1}{d!} p(\partial_1, \dots, \partial_k) q,
\end{equation}
where $p(\partial_1, \dots, \partial_k)$ denotes the formal differential operator produced by substituting coordinate differentials into $p$.\footnote{For example, if $p(x_1, x_2) = x_1x_2^2 - x_2^3$, then $p(\partial_1, \partial_2) q(x_1, x_2) = \frac{\partial^3 q}{\partial x_1 \partial x_2^2} - \frac{\partial^3 q}{\partial x_2^3}$.}
As an abbreviation we write $p(\bm\partial) \colonequals p(\partial_1, \dots, \partial_k)$.
For $p, q$ homogeneous of different degrees we also set $\langle p, q \rangle_{\circ} \colonequals 0$, and extend by linearity to define $\langle p, q \rangle_{\circ}$ for arbitrary $p, q \in \RR[z_1, \dots, z_k]$.
\end{definition}

One may verify that this is merely an inner product on the vectors of coefficients of $p$ and $q$ reweighted in a particular way; however, this particular reweighting has the following special property.
\begin{proposition}[Theorem 2.11 of \cite{Reznick-1996-HomogeneousPolynomial}]
    \label{prop:apolar-adjointness}
    Suppose $p, q, r \in \RR[z_1, \dots, z_k]^{\hom}$, with degrees $\deg(p) = a, \deg(q) = b$, and $\deg(r) = a + b$.
    Then,
    \begin{equation}
        \langle pq, r \rangle_{\circ} = \frac{a!}{(a + b)!}\langle p, q(\bm\partial)r \rangle_{\circ}.
    \end{equation}
\end{proposition}
\noindent
That is, this inner product makes multiplication and differentiation into adjoint operations.

The second idea behind spectral extension is the linear subspace of \emph{multiharmonic polynomials},
\begin{equation}
    V_{\sH}^{(d)} \colonequals \left\{ q \in \RR[z_1, \dots, z_{n - 1}]_d^{\hom}: \langle \bv_i, \bm\partial \rangle^2 q = 0 \text{ for all } i \in [n]\right\}.
\end{equation}
We call $V_{\sH}^{(d)}$ in our case the space of \emph{simplex-harmonic} polynomials.
By Proposition~\ref{prop:apolar-adjointness}, $V_{\sH}^{(d)}$ is precisely the orthogonal complement in $\RR[z_1, \dots, z_{n - 1}]^{\hom}_d$ (endowed with the apolar inner product) of the ideal generated by the polynomials $\langle \bv_i, \bm z \rangle^2$.
Thus any polynomial of $\RR[z_1, \dots, z_{n - 1}]^{\hom}_d$ decomposes into an \emph{ideal component} and a \emph{harmonic component} according to this orthogonality.
Moreover, by repeating this decomposition, we find that, for any $p \in \RR[z_1, \dots, z_{n - 1}]^{\hom}_d$, there exist $q_{T} \in V_{\sH}^{(d - 2|T|)}$ such that
\begin{equation}
    p(\bz) = \sum_{k = 0}^{\lfloor d / 2 \rfloor} \sum_{T \in \sM_k([n])} q_T(\bz) \prod_{i \in T} \langle \bv_i, \bz \rangle^2.
\end{equation}
This is sometimes called a \emph{Fischer decomposition}, and is an algebraic generalization of the following two more familiar examples where $\langle \bv_1, \bz \rangle^2, \dots, \langle \bv_n, \bz \rangle^2$ are replaced by other sequences of polynomials $f_1(\bz), \dots, f_k(\bz)$.
To lighten the notation, we write these over $m$ indeterminates $\by = (y_1, \dots, y_m)$ instead; in our case above we have $m = n - 1$.

\begin{example}[Multilinear polynomials]
    If we take $f_j(\by) = z_j^2$ for $j \in [m]$, then the associated multiharmonic subspace of $q \in \RR[y_1, \dots, y_m]^{\hom}_d$ consists of those $q$ for which $\frac{\partial^2 q}{\partial y_j^2} = 0$ for each $j$, which simply constrains $q$ to be a homogeneous multilinear polynomial.
    Indeed, the same holds for $f_j(\by) = \langle \bw_j, \by \rangle^2$ for any orthonormal basis $\bw_1, \dots, \bw_{m}$; all of the interesting behavior of the Grigoriev-Laurent pseudomoments arises precisely because of the slight overcompleteness of the simplex vectors $\bv_1, \dots, \bv_n$ as compared to an orthonormal basis.
\end{example}

\begin{example}[Spherical harmonics]
    If we consider just one polynomial $f_1(\by) = \sum_{i = 1}^{m}y_i^2$, then the associated multiharmonic subspace consists of those $q$ for which $\sum_{i = 1}^{m} \frac{\partial^2 q}{\partial y_i^2} = \Delta q = 0$, where $\Delta$ is the Laplacian operator.
    These are harmonic homogeneous polynomials, which, when restricted to the sphere, are the \emph{spherical harmonics} of classical analysis.
    The Fischer decomposition is then the decomposition of a polynomial restricted to the sphere into spherical harmonics of different degrees.
\end{example}

To describe the spectral extension prediction, we consider this decomposition applied to the polynomials $\prod_{i \in S} \langle \bv_i, \bz \rangle$ for $S \subseteq [n]$.
Let $h_{S, T} \in \RR[x_1, \dots, x_n]$ be polynomials such that $h_{S, T}(\langle \bv_1, \bz \rangle, \dots, \langle \bv_n, \bz \rangle) \in V_{\sH}^{|S| - 2|T|}$ and
\begin{equation}
    \prod_{i \in S} \langle \bv_i, \bz \rangle = \sum_{k = 0}^{\lfloor d / 2 \rfloor} \sum_{T \in \sM_k([n])} h_{S, T}(\langle \bv_1, \bz \rangle, \dots, \langle \bv_n, \bz \rangle) \prod_{i \in T} \langle \bv_i, \bz \rangle^2
\end{equation}
Then, we define
\begin{equation}
    h_{S, k}(\bx) \colonequals \sum_{T \in \sM_{(|S| - k) / 2}([n])} h_{S, T}(\bx) \in \RR[x_1, \dots, x_n]_k^{\hom}
\end{equation}
if $k \leq |S|$ and $k$ and $|S|$ are of equal parity, and $h_{S, k} = 0$ otherwise.
Finally, spectral extension predicts that, for suitable constants $\sigma_d^2$, we may take
\begin{equation}
    \label{eq:tEE-prediction-decomp-laurent-intro}
    \tEE[\bx^S\bx^T] \colonequals \sum_{d = 0}^{|S| \wedge |T|} \sigma_d^2 \cdot \bigg\langle h_{S, d}(\langle \bv_1, \bz \rangle, \dots, \langle \bv_n, \bz \rangle), h _{T,d}(\langle \bv_1, \bz \rangle, \dots, \langle \bv_n, \bz \rangle) \bigg\rangle_{\circ}.
\end{equation}
That is, $\tEE$ is a Gram matrix under the apolar inner product (extended to non-homogeneous polynomials as above) of the polynomials
\begin{equation}
    p_S(\bz) \colonequals \sum_{d = 0}^{|S|} \sigma_d h_{S, d}(\langle \bv_1, \bz \rangle, \dots, \langle \bv_n, \bz \rangle).
\end{equation}

We will show that, with a certain choice of the $\sigma_d$, \eqref{eq:tEE-prediction-decomp-laurent-intro} holds exactly for the $\tEE$ we work with here.
This both gives an interesting example where the spectral extension prediction is correct for high degrees (previously it was used to show lower bounds for random problems but only for fixed small degrees as the dimension $n$ diverges), and will allow us to derive the eigenvalues of $\bY^{(n)}$.
For further details and motivation of spectral extension, see the original references \cite{KB-2020-Degree4SK, Kunisky-2020-PositivityPreservingExtensions} or the overview in \cite{Kunisky-2021-Thesis}.

\subsection{Representation Theory of the Symmetric Group}

The upshot of our proposal above is that, in order to verify that \eqref{eq:tEE-prediction-decomp-laurent-intro} in fact holds, we will need to compute the $h_{S, d}$, which will involve computing orthogonal projections to the simplex-harmonic subspaces $V_{\sH}^{(d)}$.
While the previous works \cite{KB-2020-Degree4SK, Kunisky-2020-PositivityPreservingExtensions} using spectral extensions did this in a heuristic and approximate way for random vectors $\bv_i$, here we will use the special symmetries of the simplex to perform these computations exactly.
Namely, we will use that $V_{\sH}^{(d)}$ forms a representation of the symmetric group $S_n$.
Towards working with this representation, we first recall some general theory.

We will use standard tools such as Schur's lemma, character orthogonality, and characterizations of and formulae for characters of irreducible representations (henceforth \emph{irreps}) of the symmetric group.
See, e.g., the standard reference \cite{FH-2004-RepresentationTheory}, or \cite{Diaconis-1988-GroupRepresentationsProbability} for a more explicit combinatorial perspective.

The main combinatorial objects involved in the representation theory of $S_n$ are as follows.
A \emph{partition} $\tau = (\tau_1, \dots, \tau_m)$ of $n$ is an ordered sequence $\tau_1 \geq \cdots \geq \tau_m > 0$ such that $\sum_{i = 1}^m \tau_i = n$.
We write $\Part(n)$ for the set of partitions of $n$.
The associated \emph{Young diagram} is an array of left-aligned boxes, where the $k$th row contains $\tau_k$ boxes.
A \emph{Young tableau} of \emph{shape} $\tau$ is an assignment of the numbers from $[n]$ to these boxes (possibly with repetitions).
A tableau is \emph{standard} if the rows and columns are strictly increasing (left to right and top to bottom, respectively), and \emph{semistandard} if the rows are non-decreasing but the columns are strictly increasing.
The \emph{content} of a tableau $T$ on a diagram of $n$ boxes is the tuple $\mu$ such that $\mu_k$ is the number of times $k$ occurs in $T$.

The distinct (up to isomorphism) irreps of $S_n$ may be viewed as indexed by $\Part(n)$ and described in terms of tableaux.
We give a concrete treatment of one such construction in terms of \emph{Specht polynomials} below, which is less abstract than the usual presentation (e.g., in most of \cite{FH-2004-RepresentationTheory}) using the group algebra of $S_n$ but will be useful in the sequel.

\begin{definition}[Combinatorial representation]
    The \emph{combinatorial representation} associated to $\tau \in \Part(n)$, which we denote $U^{\tau}$, is the module of polynomials in $\RR[x_1, \dots, x_n]$ where, for each $k \in [m]$, exactly $\tau_k$ of the $x_i$ appear raised to the power $(k - 1)$ in each monomial.
    Thus, these polynomials are homogeneous of degree $\sum_{k = 1}^m (k - 1)\tau_k$.
\end{definition}

\begin{definition}[Specht module]
    \label{def:specht}
    The \emph{Specht module} associated to $\tau \in \Part(n)$, which we denote $W^{\tau}$, is the $S_n$-module of polynomials in $\RR[x_1, \dots, x_n]$ spanned by, over all standard tableaux $T$ of shape $\tau$,
    \begin{equation}
        \prod_{C} \prod_{\substack{i, j \in C \\ i < j}} (x_i - x_j),
    \end{equation}
    where the product is over columns $C$ of $T$.
    These polynomials are homogeneous of degree $\sum_C \binom{|C|}{2} = \sum_{k = 1}^m (k - 1)\tau_k$.
    We write $\chi_{\tau}$ for the character of $W^{\tau}$, and identify $W^{(n, 0)} \colonequals W^{(n)}$ and $\chi_{(n, 0)} \colonequals \chi_{(n)}$ for the sake of convenience, since we will often enumerate over $\tau$ of length at most two.
\end{definition}

\noindent
The key and classical fact concerning the Specht modules is that, over $\tau \in \Part(n)$, they are all non-isomorphic and enumerate all irreps of $S_n$.
The main extra fact we will use is the following, showing how to decompose a combinatorial representation in these irreps.

\begin{proposition}[Young's rule]
    For $\tau, \mu \in \Part(n)$, the multiplicity of $W^{\mu}$ in $U^{\tau}$ is the number of semistandard Young tableaux of shape $\mu$ and content $\tau$.
\end{proposition}

\noindent
In particular, this multiplicity is zero unless $\mu \trianglelefteq \tau$, a ``majorization'' ordering relation meaning that $\tau$ may be formed by starting with $\mu$ and repeatedly producing modified Young diagrams by moving one box at a time up and to the right.
For example, this rule tells us that there should be a copy of $W^{\tau}$ in $U^{\tau}$.
Indeed, we have already seen that $W^{\tau}$ and $U^{\tau}$ consist of homogeneous polynomials of the same degree.
Moreover, all monomials in any polynomial of $W^{\tau}$ have the same set of exponents, where the number of $x_i$ raised to the $(k - 1)$th power is the number of columns of length at least $k$ in the Young diagram of $\tau$, which is just the length of the $k$th row, $\tau_k$.
Thus $W^{\tau}$ occurs as a subspace of $U^{\tau}$, verifying the consequence of Young's rule.

Finally, we will use the following decomposition of the representation consisting of polynomials over the hypercube $\{\pm 1\}^n$ given in a recent work.
We correct a small typo present in the published version in the limits of the second direct sum below.
\begin{proposition}[Theorem 3.2 of \cite{BGP-2016-SOSHypercube}]
    \label{prop:hypercube-rep}
    Let $\RR[\{\pm 1\}^n] \colonequals \RR[x_1, \dots, x_n] / \sI$, where $\sI$ is the ideal generated by $\{x_i^2 - 1\}_{i = 1}^n$.
    Recall that $d_{\max} \colonequals \lfloor n / 2 \rfloor$.
    Then,
    \begin{equation}
        \RR[\{\pm 1\}^n] = \bigoplus_{d = 0}^{d_{\max}} \bigoplus_{k = 0}^{n - 2d} \left(\sum_{i = 1}^n x_i\right)^k W^{(n - d, d)}.
    \end{equation}
    (Note that this is not merely a statement of the isomorphism type of the irreps occurring in $\RR[\{\pm 1\}^n]$ when it is viewed as a representation of $S_n$, but an actual direct sum decomposition of the space of polynomials, where the $W^{(n - d, d)}$ are meant as specific subspaces of polynomials, per Definition~\ref{def:specht}.)
\end{proposition}

\paragraph{Characters and isotypic projection}
We recall one more general representation-theoretic idea that will play an important role in our calculations.
\begin{definition}[Isotypic component and projection]
    Let $V$ be a finite-dimensional representation of a finite group $G$, $U$ an irrep of $G$ with character $\chi$, and $W \subseteq V$ the direct sum of all irreps isomorphic to $U$ in a decomposition of $V$ into irreps.
    Then, $W$ is called the \emph{isotypic component} of $U$ in $V$, and the linear map $\bm P: V \to V$ defined by
    \begin{equation}
        \bP v \colonequals \frac{\dim(U)}{|G|}\sum_{g \in G} \chi(g) gv
    \end{equation}
    is called the \emph{isotypic projection} to $W$.
\end{definition}
\noindent
We establish some basic properties of these definitions below.
\begin{proposition} \label{prop:isotypic}
    $W$ and $\bP$ satisfy the following:
    \begin{enumerate}
        \item $W$ does not depend on the choice of decomposition of $V$ into irreps.
        \item $\bP$ is a projection to $W$ (i.e., $\bP^2 = \bP$ and the image of $\bP$ is $W$).
        \item $\bP$ is an orthogonal projection with respect to any inner product on $V$ that is $G$-invariant, i.e., satisfying $\langle gv, gw \rangle = \langle v, w \rangle$ for all $g \in G$ and $v, w \in V$.
    \end{enumerate}
\end{proposition}

\paragraph{Combinatorial interpretations of characters}
As we will be computing extensively with $\chi_{(n - d, d)}$ below, it will be useful to establish a concrete combinatorial description of the values of these characters.
For a set $A = \{a_1, \dots, a_k\} \subseteq [n]$ and $\pi \in S_n$, we write $\pi(A) \colonequals \{\pi(a_1), \dots, \pi(a_k)\}$.

\begin{definition}
    For each $0 \leq d \leq n$ and $\pi \in S_n$, let
    \begin{equation}
        c_d(\pi) \colonequals \#\left\{ A \in \binom{[n]}{a}: \pi(A) = A \right\}.
    \end{equation}
\end{definition}

\begin{proposition}
    \label{prop:two-rows-character}
    For all $1 \leq d \leq n$, $\chi_{(n - d, d)} = c_d - c_{d - 1}$, and $\chi_{(n)} = c_0 = 1$.
\end{proposition}
\noindent
We give two proofs, one using the Frobenius generating function formula for irrep characters and another using combinatorial representations.

\begin{proof}[Proof 1]
    The Frobenius formula implies that, for $\pi$ having cycles $C_1, \dots, C_k$,
    \begin{equation}
        \chi_{(n - d, d)}(\pi) = [x^d]\left\{(1 - x) \prod_{i = 1}^k (1 + x^{|C_i|})\right\}.
    \end{equation}
    Since a subset fixed by $\pi$ is a disjoint union of cycles, the product term is the generating function of the numbers of fixed subsets of all sizes:
    \begin{equation}
        \prod_{i = 1}^k (1 + x^{|C_i|}) = \sum_{d = 0}^n c_d(\pi) x^d.
    \end{equation}
    The result follows since multiplication by $(1 - x)$ makes the coefficients precisely the claimed differences.
\end{proof}

\begin{proof}[Proof 2]
    The combinatorial representation $U^{(n -d, d)}$ is the subspace of multilinear polynomials in $\RR[x_1, \dots, x_n]^{\hom}_d$.
    By Young's rule, $U^{(n -d, d)} = \bigoplus_{i = 0}^d W^{(n - i, i)}$.
    On the other hand, clearly the character of $U^{(n - d, d)}$ is $c_d$.
    Thus, $\sum_{i = 0}^d \chi_{(n - d, d)} = c_d$, and the result follows by inverting this relation.
\end{proof}

\paragraph{Partial summation of characters}

We will use the following computation of partial sums of the characters $\chi_{(n - d, d)}$ over certain sets of permutations.
The proof of this statement is somewhat involved, so we leave it to Appendix~\ref{app:class-functions}.
\begin{lemma}
    \label{lem:restricted-char-sums}
    Let $0 \leq a, b \leq d$, $A \in \binom{[n]}{a}$, $B \in \binom{[n]}{b}$, and $0 \leq k \leq a \wedge b$.
    Then,
    \begin{equation}
        \frac{1}{n!}\sum_{\substack{\pi \in S_n \\ |\pi(A) \cap B| = k}} \chi_{(n - d, d)}(\pi) = \left\{\begin{array}{ll} 0 & \text{if } a \wedge b < d, \\ (-1)^{k + |A \cap B|} \frac{\binom{d}{k}}{\binom{n}{d, d - |A \cap B|, n - 2d + |A \cap B|}} & \text{if } a = b = d. \end{array} \right.
    \end{equation}
\end{lemma}
\noindent
We note also that the special case $a = b = k$ gives the summation of the character over all $\pi$ with a specified mapping $\pi(A) = B$.

\subsection{The Simplex-Harmonic Representation}

We now discuss how $V_{\sH}^{(d)}$ fits into this framework.
First, let us see why this is a representation of $S_n$.
We have that $S_n$ acts on $\RR^{n - 1}$ by permuting the $\bv_i$ (since $\sum_{i = 1}^n \bv_i = 0$ this is well-defined); this is the irreducible ``standard representation'' of $S_n$.
The symmetric powers of this irrep give actions of $S_n$ on $\RR[z_1, \dots, z_{n - 1}]_d^{\hom}$ by likewise permuting the $\langle \bv_i, \bz \rangle$, products of which form an overcomplete set of monomials.
$V_{\sH}^{(d)}$ is an invariant subspace of this action.
The next result identifies the isomorphism type of this representation.

\begin{proposition}[Isomorphism type]
    \label{prop:harmonic-rep}
    $V_{\sH}^{(d)} \cong W^{(n - d, d)}$.
    The map $\Psi: \RR[x_1, \dots, x_n]_d^{\hom} \to \RR[z_1, \dots, z_{n - 1}]_d^{\hom}$ given by defining $\Psi(\bx^S) = \prod_{i \in S} \langle \bv_i, \bz \rangle$ and extending by linearity is an isomorphism between $W^{(n - d, d)}$ and $V_{\sH}^{(d)}$ when restricted to $W^{(n - d, d)}$.
\end{proposition}
\begin{proof}
    Let us abbreviate $W = W^{(n - d, d)}$ and $V = V_{\sH}^{(d)}$.
    We first compute the dimensions of $V$ and $W$ and show that they are equal.

    For $W$, by the hook length formula,
    \begin{align}
        \dim(W)
        &= \frac{n!}{d! \cdot (n - d + 1) \cdots (n - 2d + 2) \cdot (n - 2d)!} \nonumber \\
        &= \frac{n! (n - 2d + 1)}{d! (n - d + 1)!} \nonumber \\
        &= \binom{n}{d} \cdot \frac{n - 2d + 1}{n - d + 1} \nonumber \\
        &= \binom{n}{d} - \binom{n}{d - 1}.
    \end{align}
    (The same also follows by evaluating $\chi_{(n - d, d)}$ on the identity using the formula from Proposition~\ref{prop:two-rows-character}.)

    For $V$, we note that $V$ is isomorphic to the subspace of $\Sym^d(\RR^n)$ consisting of symmetric tensors that are zero at any position with a repeated index and have any one-dimensional slice summing to zero.
    The tensors satisfying the first constraint have dimension $\binom{n}{d}$, and there are $\binom{n}{d - 1}$ one-dimensional slices.
    We verify that these slice constraints are linearly independent: they may be identified with the vectors $\ba_S \in \RR^{\binom{[n]}{d}}$ for $S \in \binom{[n]}{d - 1}$ with entries $(\ba_S)_T = \One\{S \subseteq T\}$.
    These vectors satisfy
    \begin{equation}
        \langle \ba_S, \ba_{S^{\prime}} \rangle = \left\{\begin{array}{ll} n - d + 1 & \text{if } S = S^{\prime}, \\ 1 & \text{if } |S \cap S^{\prime}| = d - 2, \\ 0 & \text{otherwise}. \end{array}\right.
    \end{equation}
    Therefore, their Gram matrix is equal to $(n - d + 1)\bm I_{\binom{n}{d - 1}} + \bA$, where $\bA$ is the adjacency matrix of the Johnson graph $J(n, d - 1)$.
    Its most negative eigenvalue is equal to $-\min(d - 1, n - d + 1) = -(d - 1)$ (see, e.g., Section 1.2.2 of \cite{BV-SRG}), the equality following since $d \leq n / 2$.
    The Gram matrix of the $\ba_S$ is therefore positive definite.
    Thus the $\ba_S$ are linearly independent, and $\dim(V) = \binom{n}{d} - \binom{n}{d - 1} = \dim(W)$.

    Therefore, to show $V \cong W$ it suffices to show that one of $V$ or $W$ contains a copy of the other.
    We show that $V$ contains a copy of $W$.
    Recall from Definition~\ref{def:specht} that $W$ is the subspace of $\RR[x_1, \dots, x_n]^{\hom}_d$ spanned by
    \begin{align}
        \prod_{a = 1}^d (x_{i_a} - x_{j_a}) &\text{ for } i_1, \dots, i_{n - d}, j_1, \dots, j_d \in [n] \text{ distinct and satisfying} \nonumber \\
        &\text{ } i_1 <  \cdots < i_{n - d}, \, \, j_1 < \cdots < j_d,\text{ and } \, i_a < j_a \text{ for } 1 \leq a \leq d.
        \label{eq:specht-def}
    \end{align}
    Note that $\ker(\Psi)$ is the ideal generated by $x_1 + \cdots + x_n$, and therefore is an invariant subspace of the $S_n$ action.
    Since $W$ is also an invariant subspace, and is irreducible, if $W$ intersected $\ker(\Psi)$ non-trivially then $W$ would be contained in $\ker(\Psi)$, which is evidently not true (for instance, none of the basis elements in \eqref{eq:specht-def} map to zero).
    Thus $\Psi$ is an isomorphism on $W$, so it suffices to show that $\Psi(W) \subseteq V$.
    Indeed, all polynomials of $\Psi(W)$ also belong to $V$: writing $\bM = \Gram(\bv_1, \dots, \bv_n)$, for any basis element and $k \in [n]$,
    \begin{align}
      &\langle \bv_k, \bm \partial \rangle^2 \prod_{a = 1}^d(\langle \bv_{i_a}, \bz \rangle - \langle \bv_{j_a}, \bz \rangle) \nonumber \\
      &\hspace{1cm} = \sum_{\{a, b\} \in \binom{[n]}{2}} (M_{k, i_a} - M_{k, j_a})(M_{k, i_b} - M_{k, j_b}) \prod_{c \in [d] \setminus \{a, b\}}(\langle \bv_{i_c}, \bz \rangle - \langle \bv_{j_c}, \bz \rangle),
    \end{align}
    and since the $i_a$ and $j_a$ are all distinct while all off-diagonal entries of $\bM$ are equal, one of the two initial factors in each term will be zero.
    Thus, $W \cong \Psi(W) \subseteq V$, and by counting dimensions $V \cong W$.
\end{proof}

The following result shows that there are no other copies of irreps of this isomorphism class in $\RR[z_1, \dots, z_{n - 1}]_d^{\hom}$.
As a consequence, we will later be able to compute orthogonal projections to $V_{\sH}^{(d)}$ using the isotypic projection.
\begin{proposition}[Multiplicity]
    \label{prop:harmonic-multiplicity}
    $V_{\sH}^{(d)}$ has multiplicity one in $\RR[z_1, \dots, z_{n - 1}]_d^{\hom}$.
\end{proposition}
\begin{proof}
    We show the stronger statement that one is the multiplicity of $V_{\sH}^{(d)}$ in $\RR[x_1, \dots, x_n]^{\hom}_d$, which contains a copy of $\RR[z_1, \dots, z_{n - 1}]_d^{\hom}$ as the quotient by the ideal generated by $x_1 + \cdots + x_n$.
    We use that $\RR[x_1, \dots, x_n]_d^{\hom}$ admits a decomposition into invariant subspaces $\widetilde{U}^{\tau}$ over $\tau \in \Part(d)$, $\RR[x_1, \dots, x_n]_d^{\hom} = \bigoplus_{\tau \in \Part(d)} \widetilde{U}^{\tau}$, where $\widetilde{U}^{\tau}$ consists of polynomials whose monomials have their set of exponents equal to the numbers appearing in $\tau$.

    We claim that each $\widetilde{U}^{\tau}$ with $\tau = (\tau_1, \dots, \tau_m)$ is isomorphic to the combinatorial representation $U^{(n - m, f_1, \dots, f_{\ell})}$, where the $f_i$ give, in descending order, the frequencies of numbers appearing among the $\tau_i$. This is because $\widetilde{U}^{\tau}$ is the space of polynomials where each monomial contains $m$ variables with the collection of exponents equal to $\tau$, while $U^{\tau}$ is the same but with the collection of exponents equal to some $\tau^{\prime}$.
    While $\tau$ and $\tau^{\prime}$ may be different (indeed, $\tau^{\prime}$ may be a partition of some number other than $d$ into $m$ parts), the associated partition of $m$ given by the frequency of each number in either $\tau$ or $\tau^{\prime}$ is the same, namely equal to $(f_1, \dots, f_{\ell})$.
    The isomorphism type of the associated representation only depends on this ancillary partition, so $\widetilde{U}^{\tau} \cong U^{(n - m, f_1, \dots, f_{\ell})}$.

    By Young's rule, among these combinatorial representations, only $\widetilde{U}^{(1, \dots, 1)} \cong U^{(n - d, d)}$ contains a copy of $V_{\sH}^{(d)} \cong W^{(n - d, d)}$ (as, by the majorization condition, for this to happen we must have $m = d$), and it contains exactly one copy of $V_{\sH}^{(d)}$.
    We note that in this case the isomorphism $\widetilde{U}^{(1, \dots, 1)} \cong U^{(n - d, d)}$ is immediate, since both representations are just the space of multilinear polynomials of degree $d$.
\end{proof}

\section{Proof of the Grigoriev-Laurent Lower Bound}
\label{sec:laurent-pf}

Proposition~\ref{prop:hypercube-rep} gives us a means of showing that $\bY^{(n)} \succeq \bm 0$, giving a proof of Theorem~\ref{thm:laurent-pm}.
Viewing $\bY^{(n)}$ as operating on $\RR[\{\pm 1\}^n]$, since $\bY^{(n)}$ commutes with the action of $S_n$, by Schur's lemma it acts as a scalar on each irreducible subrepresentation of $\RR[\{\pm 1\}^n]$.
Since the ideal generated by $\sum_{i = 1}^n x_i$ is in the kernel of $\bY^{(n)}$, by Proposition~\ref{prop:hypercube-rep} the only possible such irreducible subrepresentations on which $\bY^{(n)}$ has a non-zero eigenvalue are $W^{(n, 0)}, W^{(n - 1, 1)}, \dots, W^{(n - d_{\max}, d_{\max})}$.
Thus it suffices to choose a non-zero element of each isotypic component, $p_i \in \bigoplus_{k = 0}^{n - 2d} \left(\sum_{i = 1}^n x_i\right)^k W^{(n - i, i)}$, and verify that $\tEE[p_i(\bx)^2] > 0$ for each $i$.

To identify such polynomials, we compute the isotypic projections of monomials.
\begin{definition}[Isotypic projection]
    \label{def:isotypic-proj}
    For each $S \in \binom{[n]}{d}$, define $h_S \in \RR[x_1, \dots, x_n]_d^{\hom}$ by
    \begin{equation}
        h_S(\bx) = \frac{\binom{n}{d} - \binom{n}{d - 1}}{n!} \sum_{\pi \in S_n} \chi_{(n - d, d)}(\pi) \bx^{\pi(S)},
    \end{equation}
\end{definition}
\noindent
Then, by Proposition~\ref{prop:isotypic}, we have that $h_S(\bx) \in \bigoplus_{k = 0}^{n - 2d} \left(\sum_{i = 1}^n x_i\right)^k W^{(n - |S|, |S|)}$, the isotypic component of $W^{(n - |S|, |S|)}$ in $\RR[\{\pm 1\}^n]$.

\begin{proposition}
    \label{prop:laurent-hS-norm}
    For any $S \in \binom{[n]}{d}$,
    \begin{equation}
        \tEE[h_S(\bx)^2] = \frac{n - 2d + 1}{n - d + 1} \prod_{i = 0}^{d - 1} \frac{n - 2i}{n - 2i - 1} > 0.
    \end{equation}
\end{proposition}
\begin{proof}
    First, since $h_S(\bx)$ is the sum of the projection of $\bx^S$ to one of the eigenspaces of $\tEE$ and an element of the kernel of $\tEE$, we have
    \begin{align}
        \tEE[h_S(\bx)^2]
        &= \tEE[h_S(\bx)\bx^S] \nonumber
        \intertext{and from here we may compute directly,}
        &= \frac{\binom{n}{d} - \binom{n}{d - 1}}{n!} \sum_{k = 0}^{d} a_{2d - 2k} \sum_{\substack{\pi \in S_n \\ |\pi(S) \cap S| = k}} \chi_{(n - d, d)}(\pi) \nonumber \\
        &= \frac{\binom{n}{d} - \binom{n}{d - 1}}{\binom{n}{d}} \sum_{k = 0}^{d} (-1)^{d - k} \binom{d}{k} a_{2d - 2k} \tag{Lemma~\ref{lem:restricted-char-sums}} \\
        &= \frac{n - 2d + 1}{n - d + 1}\sum_{k = 0}^{d} (-1)^{k} \binom{d}{k} a_{2k}.
    \end{align}
    It remains to analyze the sum.
    We view such a sum as a $d$th order finite difference, in this case a forward finite difference of the sequence $f(k) = a_{2k}$ for $k = 0, \dots, d$.
    Let us write $\Delta^a f$ for the sequence that is the $a$th forward finite difference.
    We will show by induction that
    \begin{equation}
        \Delta^a f(k) = a_{2k} \prod_{i = 0}^{a - 1}\frac{n - 2i}{n - 2k - 2i - 1}.
    \end{equation}
    Clearly this holds for $a = 0$.
    If the result holds for $a - 1$, then we have
    \begin{align}
        \Delta^a f(k)
        &= \Delta^{a - 1} f(k) - \Delta^{a - 1} f(k + 1) \nonumber \\
        &= a_{2k} \prod_{i = 0}^{a - 2}\frac{n - 2i}{n - 2k - 2i - 1} - a_{2k + 2} \prod_{i = 0}^{a - 2}\frac{n - 2i}{n - 2k - 2i - 3} \nonumber \\
        &= a_{2k} \prod_{i = 0}^{a - 2}\frac{n - 2i}{n - 2k - 2i - 1} + a_{2k} \frac{2k + 1}{n - 2k - 1} \prod_{i = 0}^{a - 2}\frac{n - 2i}{n - 2k - 2i - 3} \nonumber \\
        &= a_{2k} \prod_{i = 0}^{a - 2}\frac{n - 2i}{n - 2k - 2i - 1} \left(1 + \frac{2k + 1}{n - 2k - 1} \cdot \frac{n - 2k - 1}{n - 2k - 2a + 1}\right) \nonumber \\
        &= a_{2k}\prod_{i = 0}^{a - 2}\frac{n - 2i}{n - 2k - 2i - 1} \cdot \frac{n - 2(a - 1)}{n - 2k - 2(a - 1) - 1},
    \end{align}
    completing the induction.
    Evaluating at $a = d$ then gives
    \begin{equation}
        \sum_{k = 0}^d (-1)^{k} \binom{d}{k} a_{2k} = \Delta^d f(0) = \prod_{i = 0}^{d - 1} \frac{n - 2i}{n - 2i - 1},
    \end{equation}
    completing the proof.
\end{proof}

It is then straightforward to check that, together with some representation-theoretic reasoning, this implies the Grigoriev-Laurent lower bound.
\begin{proof}[Proof of Theorem~\ref{thm:laurent-pm}]
    Let $p(\bx) \in \RR[x_1, \dots, x_n]$.
    By Proposition~\ref{prop:hypercube-rep}, there exist $h_{d, k} \in W^{(n - d, d)}$ for $d \in \{0, \dots, d_{\max}\}$ and $k \in \{0, \dots, n - 2d + 1\}$ such that
    \begin{equation}
        p(\bx) = \sum_{d = 0}^{d_{\max}} \sum_{k = 0}^{n - 2d} \left(\sum_{i = 1}^n x_i\right)^k h_{d, k}(\bx).
    \end{equation}
    Since $\tEE$ is zero on multiples of $\sum_{i = 1}^n x_i$, its pseudomoment matrix $\bY^{(n)}$ acts as a scalar on each of the $W^{(n - d, d)}$ by Schur's lemma, and $h_{d, 0}$ for different $d$ have different degrees and thus orthogonal vectors of coefficients, we have
    \begin{equation}
        \tEE[p(\bx)^2] = \tEE\left[\left(\sum_{d = 0}^{d_{\max}} h_{d, 0}(\bx)\right)^2\right] = \sum_{d = 0}^{d_{\max}} \tEE[h_{d, 0}(\bx)^2] \geq 0
    \end{equation}
    by Proposition~\ref{prop:laurent-hS-norm}, completing the proof.
\end{proof}

\section{Pseudomoment Spectrum}

We now would like to recover the actual eigenvalues of $\bY^{(n)}$.
It may seem that we are close to obtaining the eigenvalues: since the $W^{(n - d, d)}$ are the eigenspaces of $\tEE$, it suffices to just find any concrete polynomial $p \in W^{(n - d, d)}$ that it is convenient to compute with, whereupon we will have $\lambda_{n, d} = \tEE[p(x)^2] / \|p\|^2$, where the norm of a polynomial is the norm of the vector of coefficients (not the apolar norm).
However, our computation above does not quite achieve this: crucially, $h_S(\bx)$ does \emph{not} belong to $W^{(n - d, d)}$; rather, it equals the projection of $\bx^S$ to \emph{all} copies of this irrep in $\RR[\{\pm 1\}^n]$, of which there are $n - 2d + 1$.
Since those copies that are divisible by $\sum_{i = 1}^n x_i$ are in the kernel of $\tEE$, we have actually computed $\tEE[h_S(\bx)^2] = \tEE[\what{h}_S(\bx)^2]$ where $\what{h}_S(\bx) \in W^{(n - d, d)}$ is the relevant component of $\bx^S$.
However, not having an explicit description of $\what{h}_S(\bx)$, we have no immediate way to compute $\|\what{h}_S\|^2$.

\begin{remark}
    One possible approach to implement this direct strategy is to try to take $p \in W^{(n - d,d)}$ to be one of the basis polynomials given in Definition~\ref{def:specht}.
    However, computing the pseudoexpectation of the square of such a polynomial gives an unusual combinatorial sum to which the character-theoretic tools we have developed do not seem to apply.
\end{remark}

Instead, we will use use the expression of $\bY^{(n)}$ as a Gram matrix offered by spectral extension, in particular verifying a description of the form \eqref{eq:tEE-prediction-decomp-laurent-intro} suggested above.

\subsection{Verifying Spectral Extension Characterization}

We now establish that $\tEE$ is an instance of spectral extension.
We first show that the basis of $h_S(\bx)$ achieves a block diagonalization of $\tEE$.

\begin{lemma}[Block diagonalization]
    \label{lem:block-diag-laurent}
    $\tEE[h_S(\bx)h_T(\bx)] = 0$ if $|S| \neq |T|$.
    If $S, T \in \binom{[n]}{d}$, then
    \begin{equation}
        \tEE[h_S(\bx)h_T(\bx)] = \sigma_d^2 \cdot \bigg \langle h_S(\langle \bv_1, \bz \rangle, \dots, \langle \bv_n, \bz \rangle), h_T(\langle \bv_1, \bz \rangle, \dots, \langle \bv_n, \bz \rangle) \bigg \rangle_{\circ}
    \end{equation}
    where
    \begin{equation}
        \sigma_d^2 = d! \, \left(\frac{n - 1}{n}\right)^d \, \prod_{i = 0}^{d - 1} \frac{n - 2i}{n - 2i - 1} > 0.
    \end{equation}
\end{lemma}
\begin{proof}
    For the sake of brevity, let $\bV$ be the matrix whose columns are the $\bv_i$, so that $\bV^{\top}\bz$ has entries $\langle \bv_i, \bz \rangle$.
    The first claim follows since if $|S| \neq |T|$ then $h_S(\bx)$ and $h_T(\bx)$ belong to orthogonal eigenspaces of $\tEE$.
    For the second claim, recall that $\Psi: \RR[x_1, \dots, x_n]_d^{\hom} \to \RR[z_1, \dots, z_{n - 1}]_d^{\hom}$ as defined in Proposition~\ref{prop:harmonic-rep} is an isomorphism on each eigenspace with non-zero eigenvalue of $\tEE$.
    Moreover, by Proposition~\ref{prop:harmonic-rep}, each such eigenspace is isomorphic to some $V_{\sH}^{(d)}$ and thus is irreducible.
    So, since the apolar inner product in $\RR[z_1, \dots, z_{n - 1}]_d^{\hom}$ is invariant under the action of $S_n$ (permuting the $\langle \bv_i, \bz \rangle$) and $\Psi(h_S(\bx)) = h_S(\bV^{\top}\bz)$, the result must hold with \emph{some} $\sigma_d^2 \geq 0$ (which must be non-negative by the positivity of $\tEE$).

    It remains to compute $\sigma_d^2$, which is
    \begin{equation}
        \sigma_d^2 = \frac{\tEE[h_S(\bx)^2]}{\|h_S(\bV^{\top}\bz)\|_{\circ}^2}
    \end{equation}
    for any $S \in \binom{[n]}{d}$.
    We computed the numerator in Proposition~\ref{prop:laurent-hS-norm}, so we need only compute the denominator.

    Define, for $0 \leq k \leq d$, $\beta_{d, k} \colonequals \langle (\bV^{\top} \bz)^S, (\bV^{\top} \bz)^T \rangle_{\circ}$ for any $S, T \in \binom{[n]}{d}$ with $|S \cap T| = k$ (as this value only depends on $|S \cap T|$).
    With this notation, since $h_S(\bV^{\top}\bz)$ is the apolar projection of $(\bV^{\top}\bz)^S$ to $V_{\sH}^{(d)}$ (as it is by definition the isotypic projection and by Proposition~\ref{prop:harmonic-multiplicity} $V_{\sH}^{(d)}$ has multiplicity one in $\RR[z_1, \dots, z_{n - 1}]_d^{\hom}$),
    \begin{align}
        \|h_S(\bV^{\top}\bz)\|_{\circ}^2
        &= \big \langle h_S(\bV^{\top}\bz), (\bV^{\top}\bz)^S \big \rangle_{\circ} \nonumber \\
        &= \frac{\binom{n}{d} - \binom{n}{d - 1}}{n!} \sum_{k = 0}^d \beta_{d, k} \sum_{\substack{\pi \in S_n \\ |\pi(S) \cap S| = k}} \chi_{(n - d, d)}(\pi) \nonumber \\
        &= (-1)^{d}\binom{d}{|S \cap T|}\frac{\binom{n}{d} - \binom{n}{d - 1}}{\binom{n}{d}} \sum_{k = 0}^{d} (-1)^k \binom{d}{k} \beta_{d, k}, \tag{Lemma~\ref{lem:restricted-char-sums}}
    \end{align}
    and we are left with a similar sum as in Proposition~\ref{prop:laurent-hS-norm}, but now a $d$th forward difference of the sequence $f(k) = \beta_{d, k}$.
    We note that, choosing a concrete $S$ and $T$ in the definition, we may write
    \begin{equation}
        \beta_{d, k} = \left\langle \prod_{i = 1}^d \langle v_i, z \rangle, \prod_{i = 1}^k \langle v_i, z \rangle \prod_{i = d + 1}^{2d - k} \langle v_i, z\rangle \right\rangle_{\circ}.
    \end{equation}
    Using this representation, it is straightforward to show, again by induction, that
    \begin{equation}
        \Delta^a f(k) = \left\langle \prod_{i = 1}^d \langle v_i, z \rangle, \prod_{i = 1}^k \langle v_i, z \rangle \prod_{i = d + a + 1}^{2d - k} \langle v_i, z\rangle \prod_{j = 1}^a \langle v_{d + j} - v_{k + j}, z \rangle \right\rangle_{\circ}.
    \end{equation}
    Therefore, we have
    \begin{align}
        \sum_{k = 0}^{d} (-1)^k \binom{d}{k} \beta_{d, k}
        &= \Delta^a f(0) \nonumber \\
        &= \left\langle \prod_{i = 1}^d \langle v_i, z \rangle, \prod_{i = 1}^d \langle v_i - v_{d + i}, z \rangle \right\rangle_{\circ} \nonumber
        \intertext{where the only contribution applying the product rule to the inner product is in the matching of the two products in their given order, whereby}
        &= \frac{1}{d!} \left(-1 - \frac{1}{n - 1}\right)^{d} \nonumber \\
        &= \frac{(-1)^d}{d!} \left(\frac{n}{n - 1}\right)^d,
    \end{align}
    and substituting completes the proof.
\end{proof}

The following then follows immediately since the $h_S(\bx)$ with $|S| = d$ are a spanning set of the isotypic component of $W^{(n - d, d)}$ in $\RR[\{\pm 1\}^n]$.
\begin{corollary}[Gram matrix expression]
    \label{cor:gram-mx-laurent}
    Let $h_{S, T} \in V_{\sH}^{(|S| - 2|T|)}$ be such that
    \begin{equation}
        \prod_{i \in S} \langle \bv_i, \bz \rangle = \sum_{k = 0}^{\lfloor |S| / 2 \rfloor} \sum_{T \in \sM_{k}([n])} h_{S, T}(\langle \bv_1, \bz \rangle, \dots, \langle \bv_n, \bz \rangle) \prod_{j \in T} \langle \bv_j, \bz \rangle^2,
    \end{equation}
    Define
    \begin{equation}
        h_{S, k}(\bx) \colonequals \sum_{T \in \sM_{(|S| - k) / 2}([n])} h_{S, T}(\bx) \in \RR[x_1, \dots, x_n]_k^{\hom}
    \end{equation}
    if $k \leq |S|$ and $k$ and $|S|$ are of equal parity, and $h_{S, k} = 0$ otherwise.
    Then,
    \begin{equation}
        \label{eq:tEE-prediction-decomp-laurent}
        \tEE[\bx^S\bx^T] = \sum_{d = 0}^{|S| \wedge |T|} \sigma_d^2 \cdot \bigg \langle h_{S, d}(\langle \bv_1, \bz \rangle, \dots, \langle \bv_n, \bz \rangle), h _{T,d}(\langle \bv_1, \bz \rangle, \dots, \langle \bv_n, \bz \rangle) \bigg\rangle_{\circ}.
    \end{equation}
\end{corollary}
\noindent
We note that \eqref{eq:tEE-prediction-decomp-laurent} is just \eqref{eq:tEE-prediction-decomp-laurent-intro} repeated verbatim.

\subsection{Iterated Schur Complements in Pseudomoment Matrix}
\label{sec:schur-complement}

As promised in the introduction, let us revisit Laurent's proposal of an alternate proof technique by taking iterated Schur complements in $\bY^{(n)}$, as suggested in the appendix of \cite{Laurent-2003-CutPolytopeSOS} under ``A tentative iterative proof.''
First, let us give a general description of the Schur complement operation on Gram matrices---though elementary, we have not encountered this observation in the literature and it may be of independent interest.

\begin{proposition}[Gramian interpretation of Schur complement]
    \label{prop:gramian-schur-complement}
    Suppose that $\ba_1, \dots, \ba_m,$ $\bb_1, \dots, \bb_n \in \RR^r$.
    Write $\bM \colonequals \Gram(\ba_1, \dots, \ba_m, \bb_1, \dots, \bb_n)$ as a block matrix with block sizes $m$ and $n$ along each axis:
    \begin{equation}
        \bM \equalscolon \left[\begin{array}{cc} \bM^{[1, 1]} & \bM^{[1, 2]} \\ \bM^{[2, 1]} & \bM^{[2, 2]} \end{array}\right].
    \end{equation}
    Write $\bP$ for the orthogonal projection to the orthogonal complement of the span of the $\ba_1, \dots, \ba_m$.
    Then,
    \begin{equation}
        \bM^{[2, 2]} - \bM^{[2, 1]}\bM^{[1, 1]^{\dagger}}\bM^{[1, 2]} = \Gram(\bP \bb_1, \dots, \bP \bb_n),
    \end{equation}
    where $\dagger$ denotes the Moore-Penrose matrix pseudoinverse.
\end{proposition}
\noindent
In words, the result says that the Schur complement in a Gram matrix merely keeps track of the effect of ``projecting away'' one subset of the Gram vectors from the remaining ones.
\begin{proof}
    Let $\bA$ have the $\ba_1, \dots, \ba_m$ as its columns and $\bB$ have the $\bb_1, \dots, \bb_n$ as its columns.
    Then,
    \begin{equation}
        \bM = \left[\begin{array}{cc} \bM^{[1, 1]} & \bM^{[1, 2]} \\ \bM^{[2, 1]} & \bM^{[2, 2]} \end{array}\right] = \left[\begin{array}{cc} \bA^{\top}\bA & \bA^{\top}\bB \\ \bB^{\top}\bA & \bB^{\top}\bB \end{array}\right].
    \end{equation}
    So, the Schur complement is
    \begin{align}
        \bM^{[2, 2]} - \bM^{[2, 1]}\bM^{[1, 1]^{\dagger}}\bM^{[1, 2]}
        &= \bB^{\top} \bB - \bB^{\top}\bA (\bA^{\top}\bA)^{\dagger} \bA^{\top} \bB \nonumber
        \intertext{and, recognizing the formula for the projection $\bm I - \bP = \bA (\bA^{\top}\bA)^{\dagger} \bA^{\top}$ to the span of the $\ba_1, \dots, \ba_m$,}
        &= \bB^{\top} \bB - \bB^{\top} (\bm I - \bP) \bB \nonumber \\
        &= \bB^{\top} \bP \bB \nonumber \\
        &= (\bP\bB)^{\top}(\bP\bB),
    \end{align}
    completing the proof.
\end{proof}

\begin{remark}[Other applications]
    Proposition~\ref{prop:gramian-schur-complement} is a useful intuitive guide to the geometric meaning of the Schur complement; we are surprised that it does not seem to be widely known.
    We highlight two connections of Proposition~\ref{prop:gramian-schur-complement} to other topics.
    First, taking Schur complements with respect to $1 \times 1$ diagonal blocks in a Gram matrix in some sequence describes precisely the course of the Gram-Schmidt orthogonalization procedure applied to the Gram vectors.
    Second, for a Gram matrix, the determinant equals the squared volume (Lebesgue measure) of the parallelopiped spanned by the Gram vectors.
    Thus let us write $\vol(\ba_1, \dots, \ba_m) \colonequals \sqrt{\det(\Gram(\ba_1, \dots, \ba_m))}$.
    Consider Schur's determinant identity,
    \begin{equation}
        \det(\bM) = \det(\bM^{[1, 1]}) \det(\bM^{[2, 2]} - \bM^{[2, 1]}\bM^{[1, 1]^{\dagger}}\bM^{[1, 2]}),
    \end{equation}
    in this context.
    It may be rewritten
    \begin{equation}
        \vol(\ba_1, \dots, \ba_m, \bb_1, \dots, \bb_n) = \vol(\ba_1, \dots, \ba_m) \cdot \vol(\bP\bb_1, \dots, \bP \bb_n),
    \end{equation}
    which we recognize as a generalized ``base $\times$ height'' formula for the volume of a parallelopiped.
\end{remark}

We now proceed to Laurent's proposal.
The idea is to compute Schur complements in $\bY^{(n)}$ with respect to the blocks indexed by $\binom{[n]}{d}$, in the order $d = 0, 1, \dots, d_{\max}$.
Laurent observed that for the first few $d$, each successive block belongs to the Johnson association scheme, and thus its positivity may be verified using the Bose-Mesner algebra tools mentioned in our introduction and its inverse calculated (at least in principle) in closed form.

Let us explain why these observations hold with reference to the tools we have developed so far.
We have seen that $\bY^{(n)}$ is the Gram matrix of the polynomials
\begin{equation}
    p_S(\bx) \colonequals \sum_{d = 0}^{|S|} \sigma_d h_{S, d}(\langle \bv_1, \bz \rangle, \dots, \langle \bv_n, \bz \rangle)
\end{equation}
under the apolar inner product.
Recall that, for any $S$, we have $h_{S, |S|} = h_S$, and the $h_S(\langle \bv_1, \bz \rangle, \dots, \langle \bv_n, \bz \rangle)$ with $|S| = d$ span $V_{\sH}^{(d)}$.
Thus, by Proposition~\ref{prop:gramian-schur-complement}, after $k$ steps of the procedure Laurent proposes, the matrix we are left with will have as its Gram vectors
\begin{equation}
    \sum_{d = k}^{|S|} \sigma_d h_{S, d}(\langle \bv_1, \bz \rangle, \dots, \langle \bv_n, \bz \rangle)
\end{equation}
over all $|S| \geq k$.
And, the block that we invert in the $(k + 1)$th step of the procedure will just be $\sigma_k^2$ times the Gram matrix of the $h_S(\langle \bv_1, \bz \rangle, \dots, \langle \bv_n, \bz \rangle)$ over all $|S| = k$.
From Definition~\ref{def:isotypic-proj} (of $h_S$ in terms of the character $\chi_{(n - k, k)}$) and Proposition~\ref{prop:two-rows-character} (describing the character $\chi_{(n - k, k)}$) it is then clear that this Gram matrix will always belong to the Johnson association scheme, since it is invariant under the action of $S_n$ on $\binom{[n]}{k}$.

\subsection{Proof of Theorem~\ref{thm:laurent-evals}}

Returning to our main task, we proceed to the derivation of the eigenvalues of $\bY^{(n)}$.

\begin{proof}[Proof of Theorem~\ref{thm:laurent-evals}]
    Let $\bA^{(d)} \in \RR^{\dim(V^{(d)}_{\sH}) \times \binom{[n]}{\leq d_{\max}}}$ have an isometric (under the apolar inner product) embedding of the $h_{S, d}$ as its columns.
    Then, the expression in \eqref{eq:tEE-prediction-decomp-laurent} says that
    \begin{equation}
        \bY^{(n)} = \sum_{d = 0}^{d_{\max}} \sigma_d^2 \bA^{(d)^{\top}}\bA^{(d)}.
    \end{equation}
    Define the matrix
    \begin{equation}
        \bA \colonequals \left[\begin{array}{c} \sigma_0 \bA^{(0)} \\ \vdots \\ \sigma_{d_{\max}} \bA^{(d_{\max})} \end{array}\right].
    \end{equation}
    Then, $\bY^{(n)} = \bA^{\top} \bA$, so the non-zero eigenvalues of $\bY^{(n)}$ are equal to those of $\bA\bA^{\top}$.
    We claim that $\bA\bA^{\top}$ is actually a diagonal matrix with diagonal entries having multiplicities $\dim(V_{\sH}^{(d)})$, so that, for some $\lambda_{n, 0}, \dots, \lambda_{n, d_{\max}}$ we have
    \begin{equation}
        \label{eq:block-diag}
        \bA\bA^{\top} = \left[\begin{array}{cccc} \lambda_{n, 0} \bm I_{\dim(V_{\sH}^{(0)})} & & & \\ & \lambda_{n, 1} \bm I_{\dim(V_{\sH}^{(1)})} & & \\ & & \ddots & \\ & & & \lambda_{n, d_{\max}} \bm I_{\dim(V_{\sH}^{(d_{\max})})} \end{array}\right].
    \end{equation}

    Consider first the diagonal blocks of $\bA\bA^{\top}$.
    By Schur's lemma, whenever $d^{\prime} \geq d$ and $d$ and $d^{\prime}$ have the same parity, then we have that $\{h_{S, d}\}_{S \in \binom{[n]}{d^{\prime}}} \subset V_{\sH}^{(d)}$ forms a \emph{tight frame} in $V_{\sH}^{(d)}$; that is, $\sum_S h_{S, d}h_{S, d}^{\top}$ is a multiple of the identity, where again the vectorization is interpreted to be isometric with respect to the apolar inner product, and where we note that the apolar inner product is invariant under the action of $S_n$.
    This is because the $h_{S, d}$ form a union of orbits under the action of $S_n$ on $V_{\sH}^{(d)}$, which is an irrep of $S_n$.

    Let $f_{d^{\prime}, d}$ denote the associated \emph{frame constant}, that is, the constant so that, for all $p \in V_{\sH}^{(d)}$, we have
    \begin{equation}
        \label{eq:f-frame-eqn}
        \sum_{S \in \binom{[n]}{d^{\prime}}} \langle h_{S, d}, p \rangle_{\circ} \, h_{S, d} = f_{d^{\prime}, d} \, p.
    \end{equation}
    Let $f_{d^{\prime}, d} = 0$ if $d > d^{\prime}$ or $d^{\prime}$ and $d$ have different parity.
    We then have
    \begin{equation}
        \bA^{(d)}\bA^{(d)^{\top}} = \sum_{S \in \binom{[n]}{\leq d_{\max}}} h_{S, d}h_{S, d}^{\top} = \left(\sum_{d^{\prime} = d}^{d_{\max}} f_{d^{\prime}, d}\right) \bm I_{\dim(V^{(d)}_{\sH})}.
    \end{equation}

    Next, consider the off-diagonal blocks of $\bA\bA^{\top}$, say the block indexed by some $d_1 \neq d_2$.
    We have $\bA^{(d_1)}\bA^{(d_2)^{\top}} = \sum_S h_{S, d_1}h_{S, d_2}^{\top}$.
    We may view this as a linear operator mapping $V_{\sH}^{(d_2)} \to V_{\sH}^{(d_1)}$.
    As before, this operator commutes with the actions of $S_n$ on these two spaces.
    However, since these are now two non-isomorphic irreps of $S_n$, by Schur's lemma we must have $\bA^{(d_1)}\bA^{(d_2)^{\top}} = \bm 0$.

    Thus the diagonal form proposed in \eqref{eq:block-diag} holds with eigenvalues $\lambda_{n, 0}, \dots, \lambda_{n, d_{\max}} > 0$ given by
    \begin{equation}
        \lambda_{n, d} = \sigma_d^2 \sum_{d^{\prime} = d}^{d_{\max}} f_{d^{\prime}, d} \text{ with multiplicity } \dim(V^{(d)}_{\sH}) = \binom{n}{d} - \binom{n}{d - 1}.
    \end{equation}
    In particular, $\bA\bA^{\top} \succ \bm 0$ strictly, so these are also precisely the positive eigenvalues of $\bY^{(n)}$, and the multiplicity of the zero eigenvalue of $\bY^{(n)}$ is
    \begin{equation}
        \binom{n}{\leq d_{\max}} - \sum_{d = 0}^{d_{\max}} \left(\binom{n}{d} - \binom{n}{d - 1}\right) = \binom{n}{\leq d_{\max}} - \binom{n}{d_{\max}} = \binom{n}{\leq d_{\max} - 1},
    \end{equation}
    as claimed.

    We now turn to the explicit computation of the eigenvalues.
    Let us write
    \begin{equation}
        \eta_{d^{\prime}, d}^2 \colonequals \| h_{S, d} \|_{\circ}^2 \text{ for any } S \in \binom{[n]}{d^{\prime}},
    \end{equation}
    noting that these numbers are all equal by symmetry.
    Then, the frame constants from \eqref{eq:f-frame-eqn} are
    \begin{equation}
        f_{d^{\prime}, d} = \frac{\binom{n}{d^{\prime}}}{\dim(V_{\sH}^{(d)})}\eta_{d^{\prime}, d}^2 = \frac{\binom{n}{d^{\prime}}}{\binom{n}{d} - \binom{n}{d - 1}}\eta_{d^{\prime}, d}^2.
    \end{equation}

    It remains to compute the $\eta_{d^{\prime}, d}$, which will yield the $f_{d^{\prime}, d}$ and then in turn the eigenvalues $\lambda_{n, d}$.
    We first note that, by our earlier computation in Lemma~\ref{lem:block-diag-laurent}, for any given $S \in \binom{[n]}{d}$,
    \begin{equation}
        \eta_{d, d}^2 = \|h_S\|_{\circ}^2 = \frac{\binom{n}{d} - \binom{n}{d - 1}}{\binom{n}{d}} \frac{1}{d!} \left(\frac{n}{n - 1}\right)^d.
    \end{equation}
    Therefore,
    \begin{equation}
        f_{d, d} = \frac{\binom{n}{d}}{\binom{n}{d} - \binom{n}{d - 1}} \eta_{d, d}^2 = \frac{1}{d!}\left(\frac{n}{n - 1}\right)^d.
    \end{equation}

    To compute the $\eta_{d^{\prime}, d}$ with $d^{\prime} > d$, we use that $\tEE$ itself can be used to compute the following inner products, by Corollary~\ref{cor:gram-mx-laurent}:
    \begin{equation}
        \tEE[\bx^S h_T(\bx)] = \sigma_d^2 \langle h_{S, d}(\bx), h_T(\bx) \rangle_{\circ}.
    \end{equation}
    Using that the $\{h_T(\bx)\}_{T \in \binom{[n]}{d}}$ form a tight frame with frame constant $f_{d, d}$, we have
    \begin{align}
      \eta_{d^{\prime}, d}^2
      &= \|h_{S, d}\|_{\circ}^2 \nonumber \\
      &= \frac{1}{f_{d, d}} \sum_{T \in \binom{[n]}{d}} \langle h_{S, d}(\bx), h_T(\bx) \rangle_{\circ}^2 \nonumber \\
      &= \frac{1}{\sigma_d^4 f_{d, d}} \sum_{T \in \binom{[n]}{d}} (\tEE[\bx^S h_T(\bx)])^2.
    \end{align}

    We next expand these pseudoexpectations directly:
    \begin{align}
      \tEE[\bx^S h_T(\bx)]
      &= \frac{\binom{n}{d} - \binom{n}{d - 1}}{n!} \sum_{\pi \in S_n} \chi_{(n - d, d)}(\pi) \bx^{S + \pi(T)} \nonumber \\
      &= \frac{\binom{n}{d} - \binom{n}{d - 1}}{n!} \sum_{\pi \in S_n} \chi_{(n - d, d)}(\pi) a_{d + d^{\prime} - 2|S \cap \pi(T)|} \nonumber \\
      &= \frac{\binom{n}{d} - \binom{n}{d - 1}}{n!} \sum_{k = 0}^d a_{d + d^{\prime} - 2k} \sum_{\substack{\pi \in S_n \\ |S \cap \pi(T)| = k}} \chi_{(n - d, d)}(\pi) \nonumber
      \intertext{Suppose now that $|S \cap T| = \ell$.
      Then, by Lemma~\ref{lem:restricted-char-sums} we have}
      &= \frac{\binom{n}{d} - \binom{n}{d - 1}}{\binom{n}{\ell, d - \ell, d^{\prime} - \ell, n - d - d^{\prime} + \ell}} \binom{n - 2d}{d^{\prime} - d} \binom{d}{\ell}(-1)^{\ell} \sum_{k = 0}^d \binom{d}{k}(-1)^ka_{d + d^{\prime} - 2k} \nonumber
        \intertext{The remaining sum is one we evaluated in the course of our proof of Proposition~\ref{prop:laurent-hS-norm} using finite differences. Substituting that result here then gives}
      &= \frac{\binom{n}{d} - \binom{n}{d - 1}}{\binom{n}{\ell, d - \ell, d^{\prime} - \ell, n - d - d^{\prime} + \ell}} \binom{n - 2d}{d^{\prime} - d} \binom{d}{\ell} \prod_{i = 0}^{d - 1} \frac{n - 2i}{n - d^{\prime} + d - 2i - 1} \cdot a_{d^{\prime} - d}.
    \end{align}

    Substituting this into the summation that occurs in our expression for $\eta_{d, d^{\prime}}$, we then find
    \begin{align}
      &\sum_{T \in \binom{[n]}{d}} (\tEE[\bx^S h_T(\bx)])^2 \nonumber \\
      &= \sum_{\ell = 0}^{d} \binom{d^{\prime}}{\ell} \binom{n - d^{\prime}}{d - \ell} \left(\frac{\binom{n}{d} - \binom{n}{d - 1}}{\binom{n}{\ell, d - \ell, d^{\prime} - \ell, n - d - d^{\prime} + \ell}} \binom{n - 2d}{d^{\prime} - d} \binom{d}{\ell} \prod_{i = 0}^{d - 1} \frac{n - 2i}{n - d^{\prime} + d - 2i - 1} \cdot a_{d^{\prime} - d}\right)^2 \nonumber \\
      &= a_{d^{\prime} - d}^2\left(\left(\binom{n}{d} - \binom{n}{d - 1}\right)\binom{n - 2d}{d^{\prime} - d}\prod_{i = 0}^{d - 1} \frac{n - 2i}{n - d^{\prime} + d - 2i - 1}\right)^2 \nonumber \\
      &\hspace{2.5cm} \sum_{\ell = 0}^d \binom{d^{\prime}}{\ell} \binom{n - d^{\prime}}{d - \ell} \frac{\binom{d}{\ell}^2}{\binom{n}{\ell, d - \ell, d^{\prime} - \ell, n - d - d^{\prime} + \ell}^2} \nonumber \\
      &= a_{d^{\prime} - d}^2\left(\left(\binom{n}{d} - \binom{n}{d - 1}\right)\binom{n - 2d}{d^{\prime} - d}\prod_{i = 0}^{d - 1} \frac{n - 2i}{n - d^{\prime} + d - 2i - 1}\right)^2 \nonumber \\
      &\hspace{2.5cm} \frac{d!^2 d^{\prime}! (n - d^{\prime})!(n - d - d^{\prime})! (d^{\prime} - d)!}{n!^2} \sum_{\ell = 0}^d \binom{d^{\prime} - \ell}{d^{\prime} - d} \binom{n - d - d^{\prime} + \ell}{n - d - d^{\prime}} \nonumber
        \intertext{and the remaining sum evaluates by the Chu-Vandermonde identity to}
      &= a_{d^{\prime} - d}^2\left(\left(\binom{n}{d} - \binom{n}{d - 1}\right)\binom{n - 2d}{d^{\prime} - d}\prod_{i = 0}^{d - 1} \frac{n - 2i}{n - d^{\prime} + d - 2i - 1}\right)^2 \nonumber \\
      &\hspace{2.5cm} \frac{d!^2 d^{\prime}! (n - d^{\prime})!(n - d - d^{\prime})! (d^{\prime} - d)!}{n!^2} \binom{n - d + 1}{d}.
    \end{align}

    Having reached this expression, we may substitute for $\eta_{d^{\prime}, d}^2$ and find many cancellations, obtaining
    \begin{align}
      \eta_{d^{\prime}, d}^2
      &= \frac{1}{\sigma_d^4 f_{d, d}} \sum_{T \in \binom{[n]}{d}} (\tEE[\bx^S h_T(\bx)])^2 \nonumber \\
      &= a_{d^{\prime} - d}^2 \left(\frac{n}{n - 1}\right)^d \left(\prod_{i = 0}^{d - 1} \frac{n - 2i - 1}{n - 2i - 1 - d^{\prime} + d}\right)^2 \left(\binom{n}{d} - \binom{n}{d - 1}\right)^2 \nonumber \\
    &\hspace{1cm} \frac{d! d^{\prime}! (n - d^{\prime})! (n - 2d)!^2}{n!^2 (n - d - d^{\prime})! (d^{\prime} - d)!} \binom{n - d + 1}{d}.
    \end{align}
    Then we may again substitute for $\lambda_{n, d}$ and find more cancellations, obtaining
    \begin{align}
      \lambda_{n, d}
      &= \sigma_d^2 \sum_{d^{\prime} = d}^{d_{\max}} f_{d^{\prime}, d} \nonumber \\
      &= \sigma_d^2 \sum_{d^{\prime} = d}^{d_{\max}} \frac{\binom{n}{d^{\prime}}}{\binom{n}{d} - \binom{n}{d - 1}} \eta_{d^{\prime}, d}^2 \nonumber \\
      &= n! \sum_{d^{\prime} = d}^{d_{\max}} \frac{a_{d^{\prime} - d}^2}{(n - d - d^{\prime})!(d^{\prime} - d)!} \prod_{i = 0}^{d - 1} \frac{1}{(n - 2i - 1 - d^{\prime} + d)^2}.
    \end{align}

    The formula for $\lambda_{n, 0}$ then follows immediately.
    To obtain the recursion, we compute
    \begin{align}
    \lambda_{n + 2, d + 1}
    &= (n + 2)! \sum_{d^{\prime} = d + 1}^{(n + 1) / 2} \frac{a_{n + 2, d^{\prime} - d - 1}^2}{(n - d - d^{\prime} + 1)!(d^{\prime} - d - 1)!} \prod_{i = 0}^{d} \frac{1}{(n - 2i + 2 - d^{\prime} + d)^2} \nonumber \\
    &= (n + 2)! \sum_{d^{\prime} = d}^{(n - 1) / 2} \frac{a_{n + 2, d^{\prime} - d}^2}{(n - d - d^{\prime})!(d^{\prime} - d)!} \prod_{i = 0}^{d} \frac{1}{(n - 2i + 1 - d^{\prime} + d)^2} \nonumber \\
    &= (n + 2)! \sum_{d^{\prime} = d}^{(n - 1) / 2} \frac{1}{(n + 1-d^{\prime}+ d)^2} \frac{a_{n + 2, d^{\prime} - d}^2}{(n - d - d^{\prime})!(d^{\prime} - d)!} \prod_{i = 0}^{d - 1} \frac{1}{(n - 2i - 1 - d^{\prime} + d)^2} \nonumber
    \intertext{and noting that $a_{n + 2, 2k} = \frac{n - 2k + 1}{n + 1}a_{n, 2k}$, we find}
    &= \frac{(n + 2)!}{(n + 1)^2} \sum_{d^{\prime} = d}^{(n - 1) / 2} \frac{a_{n, d^{\prime} - d}^2}{(n - d - d^{\prime})!(d^{\prime} - d)!} \prod_{i = 0}^{d - 1} \frac{1}{(n - 2i - 1 - d^{\prime} + d)^2} \nonumber \\
    &= \frac{n + 2}{n + 1}\lambda_{n, d},
    \end{align}
    completing the proof.
\end{proof}

\section*{Acknowledgments}
\addcontentsline{toc}{section}{Acknowledgments}

We thank Jess Banks for several helpful discussions about our representation-theoretic arguments.

\addcontentsline{toc}{section}{References}
\bibliographystyle{alpha}
\bibliography{main}

\appendix

\section{Character Sums: Proof of Lemma~\ref{lem:restricted-char-sums}}
\label{app:class-functions}

\begin{definition}
For $\pi \in S_n$, $0 \leq a, b \leq n$, and $0 \leq k, \ell \leq a \wedge b$, we define
\begin{align}
    f_{a, k}(\pi) &\colonequals \#\left\{A \in \binom{[n]}{a}: |\pi(A) \cap A| = k \right\}, \\
    g_{a, b, k, \ell}(\pi) &\colonequals \#\left\{A \in \binom{[n]}{a}, B \in \binom{[n]}{b}: |A \cap B| = k, |\pi(A) \cap B| = \ell\right\}.
\end{align}
\end{definition}

We will ultimately be interested in inner products with the $g_{a, b, k, \ell}$, but the following shows that these reduce to linear combinations of the $f_{a, k}$.
\begin{proposition}
    \label{prop:class-fn-g-to-f}
    For all $0 \leq k, \ell \leq a \wedge b$,
    \begin{equation}
        g_{b, a, k, \ell} = g_{a, b, k, \ell} = g_{a, b, \ell, k} = \sum_{j = 0}^a \left(\sum_{i = 0}^j \binom{j}{i} \binom{a - j}{k - i} \binom{a - j}{\ell - i}\binom{n - 2a + j}{b - k - \ell + i}\right)f_{a, j}
    \end{equation}
\end{proposition}
\begin{proof}
    The first equality holds since $|\pi(A) \cap B| = |A \cap \pi^{-1}(B)|$, and so since inversion does not change the conjugacy class of $\pi$, we have $g_{a, b, k, \ell}(\pi) = g_{b, a, k, \ell}(\pi^{-1}) = g_{b, a, k, \ell}(\pi)$.

    Suppose $A \in \binom{[n]}{a}$ with $|A \cap \pi(A)| = j$.
    Then, $B \in \binom{[n]}{b}$ with $|A \cap B| = k$ and $|\pi(A) \cap B| = \ell$ consists of some $0 \leq i \leq j$ elements of $A \cap \pi(A)$, $k - i$ elements of $A \setminus \pi(A)$, $\ell - i$ elements of $\pi(A) \setminus A$, and $b - i - (k - i) - (\ell - i) = b - k - \ell + i$ elements of $[n] \setminus A \setminus \pi(A)$.
    Thus,
    \begin{align}
      &g_{a, b, k, \ell}(\pi) \nonumber \\
    &= \sum_{A \in \binom{[n]}{a}} \#\left\{B \in \binom{[n]}{b}: |A \cap B| = k, |\pi(A) \cap B| = \ell\right\} \nonumber \\
    &= \sum_{A \in \binom{[n]}{a}} \sum_{i = 0}^{|A \cap \pi(A)|} \binom{|A \cap \pi(A)|}{i} \binom{a - |A \cap \pi(A)|}{k - i} \binom{a - |A \cap \pi(A)|}{\ell - i} \binom{n - 2a + |A \cap \pi(A)|}{b - k - \ell + i} \nonumber \\
    &= \sum_{j = 0}^a \#\left\{A \in \binom{[n]}{a}: |A \cap \pi(A)| = j\right\} \sum_{i = 0}^{j} \binom{j}{i} \binom{a - j}{k - i} \binom{a - j}{\ell - i} \binom{n - 2a + j}{b - k - \ell + i},
    \end{align}
    and the remaining cardinality is by definition $f_{a, j}(\pi)$.
\end{proof}

The following is our key combinatorial lemma, computing the inner product of $\chi_{(n - d, d)}$ with the $g_{a, b, k, \ell}$ so long as one of $a$ and $b$ is at most $d$.
\begin{proposition}
    \label{prop:class-function-ips}
    For all $0 \leq a \leq d \leq n / 2$, $a \leq b \leq n$, and $0 \leq k, \ell \leq a \wedge b$,
    \begin{equation}
        \frac{1}{n!} \sum_{\pi \in S_n} \chi_{(n - d, d)}(\pi) g_{a, b, k, \ell}(\pi) = \left\{\begin{array}{ll} 0 & \text{if } a < d, \\ (-1)^{k + \ell} \binom{d}{k}\binom{d}{\ell}\binom{n - 2d}{b - d} & \text{if } a = d. \end{array}\right.
    \end{equation}
\end{proposition}

\begin{proof}
    We first compute the inner products with the $f_{a, k}$.
    To this end, we introduce the functions
    \begin{equation}
        \label{eq:F-f}
        F_{a, j} \colonequals \sum_{k = j}^a \binom{k}{j} f_{a, k}.
    \end{equation}
    Then, we have
    \begin{align}
    F_{a, j}(\pi)
    &= \sum_{A \in \binom{[n]}{a}} \binom{|A \cap \pi(A)|}{j} \nonumber \\
    &= \sum_{A \in \binom{[n]}{a}} \sum_{C \in \binom{A}{j}} 1\{\pi(C) \subseteq A\} \nonumber \\
    &= \sum_{C \in \binom{[n]}{j}} \sum_{B \in \binom{[n] \setminus C}{a - j}} 1\{\pi(C) \subseteq C \cup B\} \nonumber \\
    &= \sum_{C \in \binom{[n]}{j}} \binom{n - 2j + |\pi(C) \cap C|}{a - 2j + |\pi(C) \cap C|} \nonumber  \\
    &= \sum_{i = 0}^j \binom{n - 2j + i}{a - 2j + i} f_{j, i}(\pi).
    \end{align}
    On the other hand, we may invert the relation \eqref{eq:F-f} (this ``inversion of Pascal's triangle'' follows from the binomial coefficients giving the coefficients of the polynomial transformation $p(x) \mapsto p(x + 1)$, whereby the inverse gives the coefficients of the transformation $p(x) \mapsto p(x - 1)$; it is also sometimes called the \emph{Euler transform}) to obtain the closed recursion
    \begin{equation}
        f_{a, k} = \sum_{j = k}^a (-1)^{j + k} \binom{j}{k} F_{a, j} = \sum_{j = k}^a (-1)^{j + k} \binom{j}{k} \sum_{i = 0}^j \binom{n - 2j + i}{a - 2j + i} f_{j, i}.
    \end{equation}
    In particular, the only non-zero term with $j = a$ is $(-1)^{a + k} \binom{a}{k} f_{a, a}$.
    We know that
    \begin{equation}
        f_{a, a} = c_a = \sum_{d = 0}^a \chi_{(n - d, d)}.
    \end{equation}
    Thus, by induction it follows that, in the character expansion of $f_{a, k}$, $\chi_{(n - d, d)}$ appears only if $a \geq d$, and when $a = d$ it appears with coefficient $(-1)^{d + k} \binom{d}{k}$.
    Thus we have
    \begin{equation}
        \frac{1}{n!} \sum_{\pi \in S_n} \chi_{(n - d, d)}(\pi) f_{a, k}(\pi) = \left\{\begin{array}{ll} 0 & \text{if } a < d, \\ (-1)^{d + k}\binom{d}{k} & \text{if } a = d. \end{array}\right.
    \end{equation}

    The first case of our claim, with $a < d$, now follows immediately from Proposition~\ref{prop:class-fn-g-to-f}.
    For the second case, with $a = d$, we proceed by induction on $n$.
    First, making a general manipulation, again by Proposition~\ref{prop:class-fn-g-to-f} we have
    \begin{align}
        &\hspace{-1cm}\frac{1}{n!} \sum_{\pi \in S_n} \chi_{(n - d, d)}(\pi) g_{a, b, k, \ell}(\pi) \nonumber \\
        &= \sum_{j = 0}^d \sum_{i = 0}^j \binom{j}{i} \binom{d - j}{k - i} \binom{d - j}{\ell - i}\binom{n - 2d + j}{b - k - \ell + i} \frac{1}{n!} \sum_{\pi \in S_n} \chi_{(n - d, d)}(\pi) f_{d, j}(\pi) \nonumber  \\
        &= \sum_{j = 0}^d (-1)^{d + j} \binom{d}{j} \sum_{i = 0}^j \binom{j}{i} \binom{d - j}{k - i} \binom{d - j}{\ell - i}\binom{n - 2d + j}{b - k - \ell + i} \nonumber
        \intertext{We start to treat the remaining sum using that $\sum_{j = 0}^d (-1)^{j} \binom{d}{j} f(j)$ gives the $d$th finite difference of the function $f$.
    In particular, for $f$ a polynomial of degree smaller than $d$, any such sum is zero.
    Furthermore, $\sum_{j = 0}^d (-1)^j \binom{d}{j}j^d = (-1)^d d!$.
    Therefore, we may continue, always applying the differencing $\Delta$ transformation to functions of the variable $j$,}
        &= \sum_{i = 0}^d \sum_{w + x + y + z = d}\binom{d}{w, x, y, z}\nonumber \\
      &\hspace{2cm} \frac{\Delta^w j^{\underline{i}} \cdot \Delta^x(d - j)^{\underline{k - i}} \cdot \Delta^y (d - j)^{\underline{\ell - i}} \cdot \Delta^z (n - 2d + j)^{\underline{b - k - \ell + i}}\, \big|_{j = 0}}{i! (k - i)! (\ell - i)! (b - k - \ell + i)!} \label{eq:chi-g-inter}
    \end{align}
    Here, in all cases the first factor, $\Delta^w j^{\underline{i}} \, \big|_{j = 0}$, will only be nonzero when $w = i$.

    Let us now first specialize to the base case $n = 2d$.
    In this case, the last factor, $\Delta^z (n - 2d + j)^{\underline{b - k - \ell + i}}\, \big|_{j = 0}$, will likewise only be nonzero when $b - k - \ell + w = z$.
    In that case, we must have $x + y = k + \ell - 2w + d - b$.
    Since in all nonzero terms $x \leq k - w$ and $y \leq \ell - w$, and $d \leq b$, we will only have a nonzero result if $d = b, x = k - w$, and $y = \ell - w$.
    In this case, we have
    \begin{align}
      \frac{1}{n!} \sum_{\pi \in S_n} \chi_{(n - d, d)}(\pi) g_{d, d, k, \ell}(\pi)
      &= (-1)^{k + \ell} \sum_{w = 0}^d \binom{d}{w, k - w, \ell - w, d - k - \ell + w} \nonumber \\
      &= (-1)^{k + \ell}\binom{d}{k} \binom{d}{\ell},
    \end{align}
    the final step following since the remaining sum counts the number of ways to choose a subset of size $k$ and a subset of size $\ell$ from $[d]$, with $w$ being the size of the intersection.
    Thus the result holds when $n = 2d$.

    Suppose now that $n > 2d$ and the result holds for $n - 1$.
    Continuing from \eqref{eq:chi-g-inter} above and completing the computation of the differences,
    \begin{align}
        &\frac{1}{n!} \sum_{\pi \in S_n} \chi_{(n - d, d)}(\pi) g_{a, b, k, \ell}(\pi) \nonumber \\
        &= \sum_{w + x + y + z = d}(-1)^{x + y}\binom{d}{w, x, y, z} \frac{1}{(k - w)! (\ell - w)! (b - k - \ell + w)!}\nonumber \\
        &\hspace{3cm} (k - i)^{\underline{x}}(d - w - x)^{\underline{k - i - x}} (\ell - i)^{\underline{y}} (d - w - x - y)^{\underline{\ell - w - y}} (b - k - \ell + w)^{\underline{z}} \nonumber \\
        &\hspace{3cm} (n - 2d + w + x + y)^{\underline{b - k - \ell + w - z}} \nonumber \\
        &= \sum_{w + x + y + z = d}(-1)^{x + y} \binom{d}{w, x, y, z} \binom{d - w - x}{k - w - x} \binom{d - w - x - y}{\ell - w - y} \binom{n - 2d + w + x + y}{b - k - \ell + w - z} \nonumber
        \intertext{Reindexing in terms of $x^{\prime} \colonequals k - w - x, y^{\prime} = \ell - w - y, z^{\prime} = b - k - \ell + w - z$, which we note must be non-negative and satisfy $x^{\prime} + y^{\prime} + z^{\prime} = b - d$, we find}
        &= \sum_{w = 0}^d \, \, \sum_{x^{\prime} + y^{\prime} + z^{\prime} = b - d}(-1)^{x + y} \binom{d}{w, k - w - x^{\prime}, \ell - w - y^{\prime}, d - k - \ell + w + x^{\prime} + y^{\prime}} \nonumber \\
        &\hspace{1.5cm} \binom{d - k + x^{\prime}}{x^{\prime}} \binom{d - k - \ell + w + x^{\prime} + y^{\prime}}{y^{\prime}} \binom{n - 2d + k + \ell - w - x^{\prime} - y^{\prime}}{z^{\prime}}. \label{eq:chi-g-sum}
    \end{align}
    We emphasize here first that $b$ appears only in the summation bounds for the inner sum, and second that we have rewritten to leave only one occurrence of $z^{\prime}$, in the final factor.

    We group the terms of the sum according to whether $z^{\prime} = 0$ or $z^{\prime} > 0$:
    \begin{align}
        S_0(b, d, k, \ell) &\colonequals \nonumber \\
      &\hspace{-1cm}\sum_{w = 0}^d \, \,  \sum_{x^{\prime} + y^{\prime} = b - d} (-1)^{x + y} \binom{d}{w, k - w - x^{\prime}, \ell - w - y^{\prime}, d - k - \ell + w + x^{\prime} + y^{\prime}} \nonumber \\
        &\hspace{4cm}\binom{d - k + x^{\prime}}{x^{\prime}} \binom{d - k - \ell + w + x^{\prime} + y^{\prime}}{y^{\prime}}, \\
       S_1(n, b, d, k, \ell) &\colonequals \nonumber \\
      &\hspace{-1cm}\sum_{w = 0}^d \, \, \sum_{x^{\prime} + y^{\prime} + z^{\prime} = b - d - 1}(-1)^{x + y} \binom{d}{w, k - w - x^{\prime}, \ell - w - y^{\prime}, d - k - \ell + w + x^{\prime} + y^{\prime}} \nonumber \\
        &\hspace{4.8cm} \binom{d - k + x^{\prime}}{x^{\prime}} \binom{d - k - \ell + w + x^{\prime} + y^{\prime}}{y^{\prime}} \nonumber \\
        &\hspace{4.8cm} \binom{n - 2d + k + \ell - w - x - y}{z^{\prime} + 1}.
    \end{align}
    Then, the sum we are interested in, that given in \eqref{eq:chi-g-sum}, is $S(n, b, d, k, \ell) \colonequals S_0(b, d, k, \ell) + S_1(n, b, d, k, \ell)$.
    Now, applying the identity $\binom{m}{a} = \binom{m - 1}{a} + \binom{m - 1}{a - 1}$ to the last factor involving $z^{\prime}$ in $S_1$, we find that
    \begin{equation}
        S_1(n, b, d, k, \ell) = S_1(n - 1, b, d, k, \ell) + S(n - 1, b - 1, d, k, \ell).
        \label{eq:S1-rec}
    \end{equation}
    Thus we have
    \begin{align}
        S(n, b, d, k, \ell)
        &= S_0(b, d, k, \ell) + S_1(n, b, d, k, \ell) \nonumber \\
        &= S_0(b, d, k, \ell) + S_1(n - 1, b, d, k, \ell) + S(n - 1, b - 1, d, k, \ell) \tag{by \eqref{eq:S1-rec}} \\
        &= S(n - 1, b, d, k, \ell) + S(n - 1, b - 1, d, k, \ell) \nonumber
          \intertext{and by the inductive hypothesis}
        &= (-1)^{k + \ell}\binom{d}{k}\binom{d}{\ell}\left( \binom{n - 2d - 1}{b - d} + \binom{n - 2d - 1}{b - d - 1}\right) \nonumber \\
        &= (-1)^{k + \ell}\binom{d}{k}\binom{d}{\ell} \binom{n - 2d}{b - d},
    \end{align}
    completing the induction.
\end{proof}

\begin{proof}[Proof of Lemma~\ref{lem:restricted-char-sums}]
    Let us write $\ell \colonequals |A \cap B|$.
    Then, using that $\chi_{(n - d, d)}$ is a class function, we may average over conjugations,
    \begin{align}
      \sum_{\substack{\pi \in S_n \\ |\pi(A) \cap B| = k}} \chi_{(n - d, d)}(\pi)
      &= \frac{1}{n!} \sum_{\sigma \in S_n} \sum_{\substack{\pi \in S_n \\ |\sigma^{-1}\pi \sigma (A) \cap B| = k}} \chi_{(n - d, d)}(\sigma^{-1}\pi \sigma) \nonumber \\
        &= \frac{1}{n!} \sum_{\sigma \in S_n} \sum_{\substack{\pi \in S_n \\ |\pi \sigma (A) \cap \sigma(B)| = k}} \chi_{(n - d, d)}(\pi) \nonumber \\
        &= \frac{1}{\binom{n}{\ell, a - \ell, b - \ell, n - a - b + \ell}} \sum_{\substack{A^{\prime} \in \binom{[n]}{a} \\ B^{\prime} \in \binom{[n]}{b} \\ |A^{\prime} \cap B^{\prime}| = \ell}} \sum_{\substack{\pi \in S_n \\ |\pi(A^{\prime}) \cap B^{\prime}| = k}} \chi_{(n - d, d)}(\pi) \nonumber \\
        &= \frac{1}{\binom{n}{\ell, a - \ell, b - \ell, n - a - b + \ell}} \sum_{\pi \in S_n} \chi_{(n - d, d)}(\pi) g_{a, b, k, \ell}(\pi),
    \end{align}
    and the result now follows from Proposition~\ref{prop:class-function-ips} upon simplifying.
\end{proof}

\end{document}